\documentclass[12pt]{amsart}
\usepackage{graphicx,amssymb,amsfonts,amsmath,amsthm,newlfont}
\usepackage{color}
\usepackage{epsfig,url}
\usepackage{numprint}
\usepackage{axodraw4j}
\usepackage{hyperref}
\usepackage{tikz,xcolor}
\usetikzlibrary{decorations.pathreplacing,decorations.markings,decorations.pathmorphing, calc}

\tikzset{
  on each segment/.style={
    decorate,
    decoration={
      show path construction,
      moveto code={},
      lineto code={
        \path [#1]
        (\tikzinputsegmentfirst) -- (\tikzinputsegmentlast);
      },
      curveto code={
        \path [#1] (\tikzinputsegmentfirst)
        .. controls
        (\tikzinputsegmentsupporta) and (\tikzinputsegmentsupportb)
        ..
        (\tikzinputsegmentlast);
      },
      closepath code={
        \path [#1]
        (\tikzinputsegmentfirst) -- (\tikzinputsegmentlast);
      },
    },
  },
  mid arrow/.style={postaction={decorate,decoration={
        markings,
        mark=at position .5 with {\arrow[#1]{stealth}}
      }}},
}

\newcommand{\CC}{{\mathbb C}}

\newcommand{\PP}{{\mathbb P}}
\newcommand{\QQ}{{\mathbb Q}}
\newcommand{\RR}{{\mathbb R}}

\newcommand{\ZZ}{{\mathbb Z}}

\newcommand{\tr}{\mathrm{tr}\,}

\newcommand{\res}{\mathrm{res}}

\theoremstyle{plain}
\newtheorem{thm}{Theorem}
\newtheorem{lem}[thm]{Lemma}

\newtheorem{quest}[thm]{Question}

\newtheorem{remark}[thm]{Remark}
\newtheorem{defn}[thm]{Definition}

\newtheorem{ex}[thm]{Example}

\newcommand{\LM}{L\mathcal{M}}

\newcommand{\GL}{\mathrm{GL}}
\newcommand{\trop}{\mathrm{trop}}

\newcommand{\To}{\longrightarrow}
\newcommand{\GC}{\mathcal{GC}}

\newcommand{\q}{/\!\!/}

\newcommand{\0}{\color{blue}{\mathsf{0}}}

\title{Non-linear geometry of multiple zeta values}

\author{Francis Brown}
\address{Francis Brown\\
All Souls College, Oxford, OX1 4AL, United Kingdom}
\email{francis.brown@all-souls.ox.ac.uk}

\begin{document}

\begin{abstract}   Since their rediscovery in the 1990s, multiple zeta values have become ubiquitous in many areas  of mathematics and physics. 
Their standard integral and sum representations can usually be traced  back to a single source, namely the iterated integrals on the  Riemann sphere with three punctures.  We refer to such representations as the \emph{linear}  geometry of multiple zeta values, since the denominators of the corresponding  integrands  factor completely  into linear terms.

However, there also exist equally important and entirely  distinct integral representations  for multiple zeta values  arising in  mathematics and physics,  in which  matrix determinants appear in the denominator of the integrand.  We call this the \emph{non-linear} geometry of multiple zeta values. 
These lectures trace the origins of this non-linear geometry and provide an introductory journey through a range of topics including tropical geometry, the moduli spaces of tropical curves, Feynman integrals  in quantum field theory, the general linear group of integer matrices,  and the reduction theory of quadratic forms. In doing so,
we  propose a  geometric framework for multiple zeta values based on such non-linear, determinantal representations  and set out a number of open questions for future research.

\end{abstract}

\maketitle

\section{Introduction}

There are essentially two, fundamentally different,  types of integral representations for multiple zeta values which frequently occur in mathematics and theoretical physics.   

\subsection{Linear geometry}
Classical integral representations of multiple zeta values involve \emph{linear} functions of the integration variables. 
Examples  arise in the theory of knot invariants and associators, in conformal field theory (via the holonomy of the Knizhnik-Zamolodchikov equation), in the study of the moduli spaces $\mathcal{M}_{0,n}$ of Riemann spheres with $n$ marked points, in deformation quantisation, as periods of hyperplane configurations and as iterated integrals of unipotent fundamental groups.    By now the range of subjects in which multiple zeta values appear is huge and the reader is invited to consult the comprehensive list \cite{HoffmanHomePage}. 
To illustrate what we mean by linear geometry,  consider the following  two  integrals.
\begin{eqnarray} 
I_1 & =& \int_{0\leq t_1 \leq t_2 \leq t_3 \leq 1}  \frac{dt_1}{1-t_1 }\frac{dt_2}{t_2} \frac{dt_3}{t_3-t_1}  \\
I_2 & = &  \int_{0\leq t_1 \leq t_2 \leq t_3 \leq 1}  \frac{dt_1}{1-t_1 }\frac{dt_2}{t_2} \frac{dt_3}{t_3}  = \sum_{n\geq 1}  \frac{1}{n^3} = \zeta(3)  \ .   
\end{eqnarray}
Both integrals have denominators  defined by linear forms in the $t_i$, and their domains of integration are defined by  \emph{linear inequalities}. 
Indeed, the integral $I_1$ is a period  of the moduli space $\mathcal{M}_{0,6}$ of 6 marked points on the Riemann sphere, labelled $0,t_1,t_2,t_3,1,\infty$. It is isomorphic to  a complement of  linear spaces 
\[ \{ t_i\neq 0, t_i\neq 1, \hbox{ for } 1\leq i \leq 3 \ , \       t_1 \neq t_2, t_1 \neq t_3, t_2\neq t_3\}   \]
in affine $3$ space with coordinates   $(t_1,t_2,t_3)$.  It is not \emph{a priori} an iterated integral but can be reduced to one. 
The integral $I_2$, however,  is  by definition an iterated integral on the projective line minus 3 points and a period of the unipotent fundamental groupoid $\pi_1^{\mathrm{un}}(\PP^1 \setminus \{0,1,\infty\})$ with respect to certain tangential basepoints at $0$ and $1$. 
One may prove  that $I_1 = 2I_2$. 
Most of the relations between multiple zeta values in the literature arise as equalities of linear integrals such as the one  above. 

The linear theory of multiple zeta values is well-developed. Indeed,  there is a very extensive literature studying variants of MZV's which are defined as iterated sums, which can  be converted into similar integrals of linear type.  It turns out in practice that  \emph{all the integral representations}  of multiple zeta values  of linear type mentioned above (such as $I_1$) may be reduced to iterated integrals on  the single space $\PP^1 \setminus \{0,1,\infty\}$ such as $I_2$.\footnote{For example, the  periods of $\mathcal{M}_{0,n}$ may be computed  using Panzer's hyperint \cite{Panzer}, which in particular implements the algorithm in \cite{BrENS}}

Note that, in the  `linear' setting, the weight of a multiple zeta value  is typically equal to the  number of integrations in its  integral representation. 

\subsection{Non-linear geometry}There exist  completely different types of integral representations for multiple zeta values 
which involve   non-linear, and highly singular, geometries. These are much less well-known, but examples include:
 Feynman integrals in Quantum Field Theory; canonical integrals on moduli spaces of tropical curves (metric graphs), 
 the volumes of locally symmetric spaces of quadratic forms, going back to Minkowski; and  regulators in algebraic $K$-theory. 
The last two examples are   classical and predate the modern theory of MZV's by a long time, but  only concern the single Riemann zeta values $\zeta(n)$. \footnote{One goal of these lectures is to suggest that, in fact, all multiple zeta values admit a similar non-linear geometric interpretation.  Furthermore, by the work of Siegel, the  two examples $(3)$, $(4)$ also have versions  for Dedekind zeta values of arbitrary number fields,  which therefore points to a geometric  theory of analogues of  multiple zeta values associated to number fields. }
The third  interpretation, as volumes of locally symmetric spaces, is closely related to the adelic formulae for zeta values which arise in the study of Tamagawa numbers.
In these lectures, we wish   to explain why these  examples all relate to a common underlying geometry associated to  the general linear group of integer matrices $\GL_n(\ZZ)$. 

\begin{ex}  \label{exampleintrozeta3nonlinear} Consider the following non-linear representation for $ 6\, \zeta(3)$ as a 5-dimensional integral which we shall encounter in various  contexts mentioned above:
\[  6 \zeta (3)  = \int^{\infty}_{x_1,\ldots, x_5=0} \frac{dx_1 \ldots dx_5}{\det(X)^2} \] 
where  
\[   X = \begin{pmatrix}
 x_1+x_2+ 1  & -x_1  & - x_2  \\ 
 -x_1 & x_1+x_3+x_5 & -x_3 \\
 -x_2  & -x_3  & x_2+x_3+x_4  
\end{pmatrix}
\]
The  vanishing locus $\det(X)=0$ defines a singular hypersurface in affine 5-space. 
The polynomial $\det(X)$ is irreducible of degree three, hence the term `non-linear'  
(although such polynomials are typically  of degree at most 1 in every variable $x_i$  individually).  
\end{ex}
 
In fact,  the `non-linear'  representations for MZV's which we shall consider will be convergent integrals  of the general  form 
\begin{equation}  \label{introIasdet} I =  \int_{\sigma}  \frac{N(x_1,\ldots, x_n) } { \det (X)^d } dx_1 \ldots dx_n
\end{equation} 
where $\sigma$ is defined by linear inequalities in the  variables $x_1,\ldots, x_n$,  $X$ is a matrix whose entries are polynomials (typically linear forms) in the $x_i$, and $N \in \QQ[x_1,\ldots, x_n]$ is a polynomial.
It is non-linear precisely when $\det(X)$ defines an irreducible hypersurface. 
This formulation encompasses all the non-linear representations considered in these notes.
   By contrast with the linear representations considered earlier, the number of integration variables will usually  be larger  than the weight.

The guiding perspective of these notes is that multiple zeta values admit a unified geometric interpretation as periods of a “determinantal geometry”,  governed by canonical differential forms naturally associated to families of matrices. This perspective encompasses the non-linear integral representations arising in quantum field theory, tropical geometry, and the theory of arithmetic groups. A key role will be played by certain canonical differential forms, whose associated integrals are well-defined and finite, and which provide a bridge between these different settings. The aim of these lectures is to highlight the underlying geometric unity between them.

\subsection{Contents} 
The `linear' theory of multiple zeta values is well-developed and the corresponding geometric and motivic theory (namely, the theory of the motivic fundamental groupoid of the projective line minus three points) is well-understood. My aim in these lectures is to explain a parallel `non-linear' theory. 
The notes  are divided into four  sections,  corresponding roughly to each lecture at the workshop. 
\begin{enumerate} 
\item Introduction to Feynman integrals and MZV's 
\item Tropical curves and their moduli 
\item  The graph complex and Grothendieck-Teichm\"uller group 
\item  Canonical integrals and $\GL_n(\ZZ)$. 
\end{enumerate}
The structure of these notes is as follows. We first contrast the classical linear theory of multiple zeta values with a class of non-linear integral representations arising from determinants. We then show that Feynman integrals provide a natural source of such representations, and reinterpret their domains of integration in terms of moduli spaces of tropical curves. Next, we relate these spaces to the graph complex and the Grothendieck–Teichmüller Lie algebra. Finally, we introduce canonical differential forms on spaces of matrices, which yield well-defined period integrals and relate  the entire picture to the cohomology of $\mathrm{GL}_n(\mathbb{Z})$ and algebraic $K$-theory.

Some basic familiarity with MZV's will be assumed. 
Since these are lecture notes, I have chosen to focus on simple examples,  motivating ideas, and  connections between different fields of research. It was not my intention to provide  a comprehensive survey of any particular part of mathematics, nor to  review the  literature. 
When writing up these notes I decided to  include some extra topics which are marked with a single star $*$. More advanced topics are  marked with two stars, and can be omitted.
I hope that these lectures will motivate the reader to explore the fascinating and largely unexplored non-linear geometry of multiple zeta values.
I have included a  number of open problems and conjectures which suggest, amongst other things, that multiple zeta values, and by extension, mixed Tate motives over the integers,   arise in some completely unexpected contexts. These  include some very classical  objects in mathematics, notably  the moduli spaces of curves, abelian varieties, and the general linear group of integer matrices. 
\\

\emph{Acknowledgements}.   
Some of the material in these notes was based on lectures given at the ICMU, Kyiv, in November and December 2024, as well as in Fukuoka in February 2025 at Kyushu  University. I am very grateful to both institutions for their generous hospitality and for the opportunity  to deliver these  lectures. 
This project has received funding from the European Research Council (ERC) under the European Union’s Horizon Europe programme (grant agreement No. 101167287).  For the purpose of Open Access, the  author has applied a CC BY public copyright licence to any Author Accepted Manuscript (AAM) version arising from this submission.

\begin{remark} (Iterated modular integrals) The reader may wonder, 
 in view of Zudilin's lectures on multiple modular values in this volume,  how 
 iterated integrals of Eisenstein series for the modular group $\mathrm{SL}_2(\ZZ)$ of level 1  tie in with the previous discussion. 
Indeed, Saad's theorem \cite{Saad} expresses all multiple zeta values as a linear combination of iterated integrals of Eisenstein series.  In particular, every MZV can be written  as a Lambert series, generalising a famous formula for $\zeta(3)$ due to Ramanujan.
The proof in \cite{Saad} shows that the multiple modular representations of MZV's are, by a clever change of variables, equivalent to iterated integrals on the projective line minus three points and hence are  essentially `linear' representations of MZV's, but  in   a somewhat disguised form. 
\end{remark}

\begin{remark}Iterated integrals on $\PP^1 \backslash \{0,1,\infty\}$, and more generally, period integrals on $\mathcal{M}_{0,n}$,  may also be written as an integral of the  kind \eqref{introIasdet},  involving a degenerate determinant. Consider the Vandermonde-type  matrix 
\[   V =    \begin{pmatrix} 
t_1 & t_1^2 & \ldots  & t_1^{n+1} \\
\vdots & \vdots & \ddots & \vdots \\
t_n & t_n^2 & \ldots  & t_n^{n+1} \\
 1 & 1 & \ldots & 1 \\
\end{pmatrix} 
 \]
whose determinant factorises as a product of linear factors \[\det(V) =   \prod_{i=1}^n  t_i(1-t_i) \prod_{1\leq i < j \leq n} (t_j -t_i) \ . \]
Then the  periods of the moduli space $\mathcal{M}_{0,n+3}$ referred to above (and in particular, iterated integrals on the projective line minus three points) are all  of the form 
\[ I= \int_{\sigma}  \frac{N(t_1,\ldots, t_n) }{ \det(V)^d} dt_1\ldots dt_n  \] 
where the boundary of the domain of integration $\sigma$ is defined by  linear inequalities, and $N$ is a polynomial in $t_i$ with rational coefficients.   From this perspective, the linear theory of multiple zeta values  is a degenerate version of the non-linear theory. In particular, formula \eqref{introIasdet} subsumes both. 
\end{remark}

From this perspective, the distinction between `linear' and `non-linear' geometry is not absolute: the linear theory may be viewed as a degenerate case of the determinantal framework in which the relevant determinant factors completely into linear terms. In particular, the essential difference between the two settings lies not in the presence of determinants per se, but in whether the corresponding hypersurface decomposes into a union of hyperplanes or remains irreducible.

A natural question, which will reappear throughout these notes, is whether there exists a direct geometric  relationship between these two types of representations.

\section{Feynman integrals and MZV's}  \label{sect:FeynmanMZVs}

There are many comprehensive accounts of Feynman integrals in the literature. The point of view taken in \cite{Panzer, Schnetz} is particularly close to the one presented here. In my lecture I gave a self-contained and purely mathematical definition of graph polynomials and Feynman residues, which are an important  class of periods  of practical use in high-energy physics. See \S\ref{sect:physics} for additional motivation from physics.

\subsection{Graph polynomials}  \label{sect: GraphPolys}
Let $G$ be a finite, connected graph whose set of edges we denote by  $E_G$, and whose set of vertices by $V_G$. To each edge $e\in E_G$  we associate a variable  $x_e$ and define  the \emph{graph polynomial}
\[ \Psi_G \in \ZZ [x_e, e\in E_G]\]
to be the following sum over monomials:
\[ \Psi_G = \sum_{\substack{T\subset G \\ \hbox{T spanning tree} }} \prod_{e\notin T } x_e \ . \]
It was introduced by Kirchhoff in the mid 19th-century for his  study of electrical circuits. 
The  sum ranges over the set of spanning trees of $G$, which are defined to be its subgraphs which are trees, i.e., which  are connected and simply-connected (have no closed cycles), and which span the set of vertices of $G$, i.e., which satisfy $V_T =V_G$. 

\begin{ex} \label{ex: sunrise} Consider the sunrise graph with two vertices and three edges (left).

\begin{center}
\begin{tikzpicture}
% Draw a black circle
  \draw[black, thick] (0,0) circle (1cm);
  % Draw the horizontal diameter
  \draw[black, thick] (-1,0) -- (1,0);
  \filldraw[black] (0,0)  circle (0pt) node[anchor=south]{\small $2$};
    \filldraw[black] (0,1)  circle (0pt) node[anchor=south]{\small $1$};
      \filldraw[black] (0,-1)  circle (0pt) node[anchor=south]{\small $3$};
        \filldraw[black] (-1.5,-1)  circle (0pt) node[anchor=south]{$G$};
 % Draw the black upper semi-circle
\fill[black] (-1,0) circle (2pt);
\fill[black] (1,0) circle (2pt);
 
\draw[black, thick] (6,0) arc[start angle=0, end angle=180, radius=1cm];
% Draw the small filled black circles at endpoints
\fill[black] (6,0) circle (2pt);
\fill[black] (4,0) circle (2pt);
 \filldraw[black] (5,1)  circle (0pt) node[anchor=south]{\small $1$};

\draw[black, thick] (6,0) arc[start angle=0, end angle=180, radius=1cm];
% Draw the small filled black circles at endpoints
\fill[black] (6,0) circle (2pt);
\fill[black] (4,0) circle (2pt);
 \filldraw[black] (8,0)  circle (0pt) node[anchor=south]{\small $2$};
  \draw[black, thick] (7,0) -- (9,0);
\fill[black] (7,0) circle (2pt);
\fill[black] (9,0) circle (2pt);
   \filldraw[black] (11,-1)  circle (0pt) node[anchor=south]{\small $3$};
 \draw[black, thick] (12,0) arc[start angle=0, end angle=-180, radius=1cm];
\fill[black] (10,0) circle (2pt);
\fill[black] (12,0) circle (2pt);
 
\end{tikzpicture}
\end{center} 

\vspace{0.1in}
\noindent 
It has three spanning trees, shown on the right. The  complementary set of edges for each spanning tree are $\{2,3\}$, $\{1,3\}$ and $\{1,2\}$. The graph polynomial is therefore
\[ \Psi_G = x_1x_2 +x_1x_3+x_2x_3\] 

\end{ex}

\begin{ex}  \label{ex: dunce} Consider the following graph with 3 vertices and 4 edges. 
It has the five spanning trees depicted on the right, whose 
complementary sets of edges  are, from left to right: $\{3,4\}$, $\{2,4\}$, $\{1,4\}$, $\{2,3\}$,  $\{1,3\}$.  
 Its graph polynomial is:
\[ \Psi_G = x_3 x_4 + x_2 x_4+x_1x_4+x_2x_3  + x_1 x_3\]

\begin{figure}[h!] 
\begin{tikzpicture}
\filldraw[black] (-2,0.5)  circle (0pt) node[anchor=south]{\small $1$};
\filldraw[black] (-2,-0.5)  circle (0pt) node[anchor=south]{\small $2$};
\filldraw[black] (-0.9,0)  circle (0pt) node[anchor=east]{\small $3$};
\filldraw[black] (-0.7,0)  circle (0pt) node[anchor=west]{\small $4$};
  \fill[black] (-1,1) circle(2pt);
  \fill[black] (-1,-1) circle(2pt);
   \fill[black] (-3,0) circle(2pt);
  \draw[black,thick] (-1,1) .. controls (-0.6,0.4) and (-0.6, -0.4) .. (-1,-1);
  \draw[black,thick] (-1,1) .. controls (-1.5,0.4) and (-1.5, -0.4)  .. (-1,-1);
  \draw[black,thick] (-3,0) -- (-1,-1);
   \draw[black,thick] (-3,0) -- (-1,1);
\fill[black] (3,0.5) circle(1pt);
\fill[black] (3,-0.5) circle(1pt);
\fill[black] (2,0) circle(1pt);
%\draw[black,thick] (3,0.5) .. controls (3.2,0.2) and (3.2,-0.2) .. (3,-0.5);
%\draw[black,thick] (3,0.5) .. controls (2.75,0.2) and (2.75,-0.2) .. (3,-0.5);
\draw[black,thick] (2,0) -- (3,-0.5);
\draw[black,thick] (2,0) -- (3,0.5);

\fill[black] (5,0.5) circle(1pt);
\fill[black] (5,-0.5) circle(1pt);
\fill[black] (4,0) circle(1pt);
%\draw[black,thick] (5,0.5) .. controls (5.2,0.2) and (5.2,-0.2) .. (5,-0.5);
\draw[black,thick] (5,0.5) .. controls (4.75,0.2) and (4.75,-0.2) .. (5,-0.5);
%\draw[black,thick] (4,0) -- (5,-0.5);
\draw[black,thick] (4,0) -- (5,0.5);
% Duplicate and move 2 cm to the right
  \fill[black] (7,0.5) circle(1pt);
  \fill[black] (7,-0.5) circle(1pt);
  \fill[black] (6,0) circle(1pt);
 % \draw[black,thick] (7,0.5) .. controls (7.2,0.2) and (7.2,-0.2) .. (7,-0.5);
  \draw[black,thick] (7,0.5) .. controls (6.75,0.2) and (6.75,-0.2) .. (7,-0.5);
  \draw[black,thick] (6,0) -- (7,-0.5);
%  \draw[black,thick] (6,0) -- (7,0.5);
  % Duplicated and moved 2cm to the right
  % Original Figure
  \fill[black] (9,0.5) circle(1pt);
  \fill[black] (9,-0.5) circle(1pt);
  \fill[black] (8,0) circle(1pt);
  \draw[black,thick] (9,0.5) .. controls (9.2,0.2) and (9.2,-0.2) .. (9,-0.5);
%  \draw[black,thick] (9,0.5) .. controls (8.75,0.2) and (8.75,-0.2) .. (9,-0.5);
%  \draw[black,thick] (8,0) -- (9,-0.5);
  \draw[black,thick] (8,0) -- (9,0.5);

  % Duplicated Figure
  \fill[black] (11,0.5) circle(1pt);
  \fill[black] (11,-0.5) circle(1pt);
  \fill[black] (10,0) circle(1pt);
  \draw[black,thick] (11,0.5) .. controls (11.2,0.2) and (11.2,-0.2) .. (11,-0.5);
%  \draw[black,thick] (11,0.5) .. controls (10.75,0.2) and (10.75,-0.2) .. (11,-0.5);
  \draw[black,thick] (10,0) -- (11,-0.5);
%  \draw[black,thick] (10,0) -- (11,0.5);
\filldraw[black] (2.5,0.28)  circle (0pt) node[anchor=south]{\tiny $1$};
\filldraw[black] (2.5,-0.8)  circle (0pt) node[anchor=south]{\tiny $2$};
\filldraw[black] (4.98,-0.25)  circle (0pt) node[anchor=south]{\tiny $3$};
\filldraw[black] (6.98,-0.25)  circle (0pt) node[anchor=south]{\tiny $3$};
\filldraw[black] (9.32,-0.25)  circle (0pt) node[anchor=south]{\tiny $4$};
\filldraw[black] (11.32,-0.25)  circle (0pt) node[anchor=south]{\tiny $4$};
\filldraw[black] (4.5,0.28)  circle (0pt) node[anchor=south]{\tiny $1$};
\filldraw[black] (8.5,0.28)  circle (0pt) node[anchor=south]{\tiny $1$};
\filldraw[black] (6.5,-0.8)  circle (0pt) node[anchor=south]{\tiny $2$};
\filldraw[black] (10.5,-0.8)  circle (0pt) node[anchor=south]{\tiny $2$};
\end{tikzpicture}

\end{figure}

\end{ex}

\begin{ex} \label{exloopgraphs}
The reader may like to check that the graph  with $n$ vertices and $n$ edges which form an $n$-gon satisfies
$\Psi_G = x_1+ x_2 +  \ldots + x_n$. 
\end{ex} 
For any graph $G$, the polynomial $\Psi_G$ is homogeneous of degree equal to $h_G$, the number of independent cycles of $G$. The number $h_G$ is  known as the  (first) Betti number of $G$, and 
in physics  is  called the `loop number' of $G$. We shall use the same terminology.

\subsection{Feynman residues} We shall suppose throughout  this section that the graph $G$ has twice as many edges as loops:
\begin{equation} \label{edgestwiceloops} 
|E_G | = 2 h_G
\end{equation}
We shall number the edges of $G$ from $1$ to $n = |E_G|$, and set $x_n=1$. 

\begin{defn} The Feynman residue is defined  by the integral
\begin{equation}  \label{Iresdef} 
I^{\res}_G =  \int_{0 \leq x_1, x_2, \ldots, x_{n-1} \leq \infty}  \frac{dx_1\ldots dx_{n-1} }{ \Psi^2_G\Big|_{x_n=1} } 
\end{equation} 
where $\Psi_G\big|_{x_n=1}$ denotes  $\Psi_G(x_1,\ldots, x_{n-1}, 1)$. 
\end{defn} 

The Feynman integral is often infinite, but one may show that it is finite if and only if $G$ is \emph{sub-divergence free}. This means that  there exists no strict subgraph $\gamma \subsetneq G$  (where $\gamma$ is defined by any subset of edges $E_{\gamma} \subsetneq E_{G}$ and $V_{\gamma}\subset V_{G}$ are their endpoints) such that 
\begin{equation}  \label{divergentsubgraph} 
|E_{\gamma}| \leq  2 h_{\gamma}  \ .
\end{equation}  
A subgraph $\gamma$ satisfying \eqref{divergentsubgraph} is called a divergent subgraph (or `sub-divergence').  

\begin{remark} The value of the integral does not depend on the ordering of the edges of the graph, i.e.,  which  edge variable we set to be equal to $1$. In fact, the Feynman residue is an example of a projective integral
\[  I^{\res}_G =  \int_{\sigma_G} \frac{\Omega_{G}}{\Psi_G^2}  \] 
where the domain of integration
\[ \sigma_G = \{ (x_1:\ldots :x_n) \in \PP^{n-1} (\RR) : x_i \geq 0\}\]
 is the positive coordinate simplex in projective space, and 
\[ \Omega_G = \sum_{i=1}^n (-1)^i x_i dx_1 \wedge \ldots \wedge \widehat{dx_i} \wedge \ldots \wedge dx_n\ .\]
The fact that the integral $I^{\res}_G$ is projective means that the integrand is homogeneous of degree $0$ (which follows from the assumption \eqref{edgestwiceloops}), and that it is invariant under contraction with the Euler vector field $\sum_{i=1}^n x_i \frac{\partial}{\partial x_i}$. This is in turn equivalent to the fact that the  integrand  is the homogenisation 
of a differential form in affine coordinates. The point about a projective integral is that it may be computed by restricting to any affine chart, of which $x_n=1$ is just one possibility. 
Doing so replaces $\Omega_G$ with $\Omega_G\big|_{x_n=1}= (-1)^n  dx_1\ldots dx_{n-1}$, and $\sigma_G\cap \{x_n=1\}$ becomes the infinite hypercube $[0,\infty]^{n-1}$, and thus we retrieve the formula \eqref{Iresdef}, for a suitable choice of orientation on the domain of integration.   In this setting, we typically orient  $\sigma_G$ so that the resulting integral is positive.  
 \end{remark} 
 
 \begin{ex}  Consider the one loop graph with 2 edges: 
 \begin{figure}[h] \begin{center}
 {\includegraphics[width=3cm]{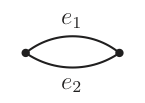}} 
\end{center}
\end{figure} 

\noindent 

We have $\Psi_G = x_1 + x_2$ and hence
  \[ I^{\res}_G = \int_{0}^{\infty}  \frac{dx_1}{(x_1+1)^2} = \left[ \frac{-1}{x_1+1} \right]^{\infty}_0 =1 \ .\] 
 \end{ex}
 
 \begin{ex} The Feynman residue for the graph  in example  \ref{ex: dunce} is infinite. This is  because the graph $G$ contains a divergent subgraph $\gamma$ whose edge set are  the edges $3,4$. It has one loop and two edges, and thus satisfies the condition \eqref{divergentsubgraph}. Similarly, the sunrise  graph example \ref{ex: sunrise} has three divergent subgraphs. 
 \end{ex} 

 \subsection{Examples of Feynman residues and summary of known results} \label{sect: ExFeynmanSummary}
 Below are examples of the first convergent (i.e.,  subdivergence-free) Feynman graphs. Their residues, which are written below each graph,   can be computed efficiently using  Hyperint \cite{Panzer}. 
\begin{figure}[h] \begin{center}
 {\includegraphics[width=10cm]{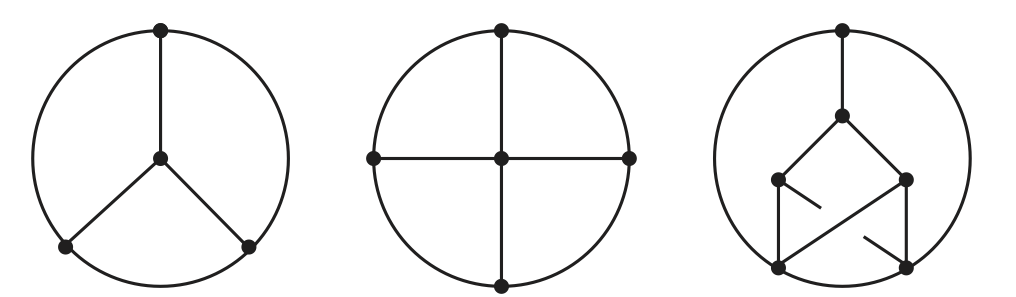}} 
 \[ I_G: \qquad \quad 6 \zeta(3) \qquad   \qquad \qquad 20 \zeta(5)  \quad \qquad \qquad 36 \zeta(3)^2  \qquad \qquad  \quad \] 
\end{center}
\end{figure}
The graph on the left is computed by hand in \cite{BrMassless} using polylogarithm functions, and illustrates how the algorithm works by replacing the non-linear geometry of the Feynman integral with a sequence of  linear fibrations with respect to each edge variable in turn.

The following graph is the complete bipartite graph $K_{3,4}$.

\begin{figure}[h] \begin{center}
 {\includegraphics[width=3cm]{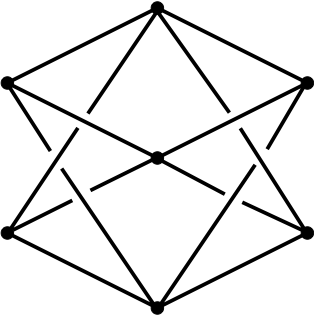}} 
\end{center}
\end{figure}

\noindent Its residue, first computed by Broadhurst and Kreimer, equals
 $$ I^{\res}_{K_{3,4}}=\textstyle{ {27\over 5} \zeta(5,3) + {45\over 4} \zeta(5)\zeta(3) -{261\over 20} \zeta(8)}$$
 which involves a non-trivial multiple zeta value of weight $8$. This graph was historically important since it was one of the very first appearances of multiple zeta values in physics. It not only  changed expectations for the numbers which could arise as Feynman integrals in physics, but was a motivating example for the renewed study of MZV's in mathematics in the 1990's after a hiatus of around two and a half centuries since their discovery by  Euler  in the 18th century.

The integrals  $I_G$ are difficult  to compute by elementary methods: in the first example, $\Psi_G$  is a polynomial of degree $3$ in $6$ variables, and has 16 terms.   

Here follows a  quick overview of known results:

\begin{itemize} \setlength{\itemsep}{0.05in}
\item Consider the wheel graphs $W_n$ with $n\geq 3$ spokes. 

\begin{figure}[h]
\begin{tikzpicture}
  % Draw the outer circle
 \draw[black, thick] (0,0) circle(2cm); \fill[black, thick] (0,0) circle(2pt); \% Draw the spokes \foreach \angle in {0, 72, 144, 216, 288} { \draw[black, thick] (0,0) -- (\angle:2cm); \fill[black, thick] (\angle:2cm) circle(2pt); 
  }
   \filldraw[black] (-3,-1)  circle (0pt) node[anchor=south]{$W_5$};
\end{tikzpicture}
\end{figure}

They are one of the few families of graphs for which the  residues are  known  explicitly. 
One may show that for all $n\geq 3$,
\begin{equation} \label{eqn: wheelresidues} I^{\res}_{W_{n+1}} = \binom{2n}{n} \, \zeta(2n-1)  \end{equation}
and thus every odd single zeta value arises as a Feynman residue. However, in many physical contexts, the wheel graphs are considered `unphysical' for $n\geq 5$ owing to the presence of the central vertex of large degree $n$ (in the standard model of physics as currently understood, at most four particles can interact at a point).

\item  The zig-zag graphs, which are  obtained by connecting the opposite vertices of a linear chain of triangles by an edge,  form an infinite family with the  property that each vertex has degree at most 4 (an important physical constraint, as mentioned above).  Below is pictured the zig-zag graph $Z_5$ with 5 loops. 
\begin{figure}[h] \begin{center}
 {\includegraphics[width=9cm]{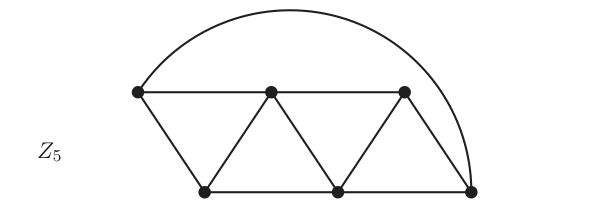}} 
\end{center}
\end{figure}

    It was shown in \cite{ZigZagGraphs} that for $n\geq 3$, 
\[ I^{\res}_{Z_n}= 4 \frac{ (2n-2)!}{ n! (n-1)!  }  \left( 1  - \frac{1- (-1)^n}{2^{2n-3}} \right) \zeta (2n-3)  \ . \]
The proof makes use of the theory of single-valued polylogarithms and leads to a formula for $I^{\res}_{Z_n}$ in terms of the multiple zeta values $\zeta(2,\ldots, 2,3, 2,\ldots,2 )$. 
The numerical values of these residues are by far the largest at each loop order (in other words, it is conjectured that $I^{\res}_G\leq I^{\res}_{Z_n}$ for any subdivergence-free graph $G$ with $n$ loops, with equality if and only if $G$ is isomorphic to $Z_n$ \cite{HeppBounds}.)

\item Not all multiple zeta values  occur as Feynman residues. To establish this one may define a motive \cite{BEK} and motivic period \cite{brownnotesmot} associated to the integrals $I^{\res}_G$. Then, using the theory of weights in mixed Hodge theory, one sees  that no $I^{\res}_G$ 
can give rise to $\zeta(2)$, simply because there exists no convergent graph  satisfying \eqref{edgestwiceloops} with the requisite number of edges. One may deduce, using the so-called coaction principle (an application of motivic Galois theory to Feynman integrals \cite{Cosmic}),  that an infinite sequence of periods which includes $\zeta(3)\zeta(2), \ldots, \zeta(2n-1)\zeta(2), \ldots $ cannot arise as Feynman residues in a precise sense.  In brief, the set of Feynman residues contains large `holes' when compared to the ring of all  MZV's.

\item  There exist some large infinite families of graphs (typically built out of triangles) whose residues  can be proven to be multiple zeta values \cite{PeriodsFeynman}. There are  others (`constructible' graphs) which are in the subring of single-valued multiple zeta values \cite{schnetz2014graphical}.  The number of graphs which can be proven to have  MZV-residues of a given weight  grows  much faster  than the dimension of the space of MZV's in that weight, which implies the existence of many relations between their residues.  To my knowledge these have not been studied. However,  the full class of graphs whose residues are MZV's appears to be even larger, but there is no complete characterisation of this class of graphs at the time of writing. 

\item Not all Feynman residues are expected to be MZV's.  If one assumes a version of the Grothendieck period conjecture, then there are explicit examples of  graphs \cite{BrownSchnetz} whose residues are algebraically independent over the ring of MZV's.
\end{itemize}

 In view of this very rich structure, one can say  that Feynman residues  play the role of a system of fundamental constants in quantum field theory, and play  a similar role to  multiple zeta values  in algebraic geometry. 
\subsection{Some  identities between Feynman residues}
Here follows an incomplete list of known algebraic identities between Feynman residues. 
\begin{enumerate}
 \setlength{\itemsep}{8pt} % Adjust the space between items
  \setlength{\parskip}{5pt} % Adjust the space between paragraphs
  \setlength{\topsep}{15pt} % Adjust the space above and below the list
\item (Two-vertex join). Feynman graphs and their residues admit the following multiplicative structure.
Suppose that we have two graphs $G_1,G_2$ and let $e_1\in E_{G_1}$, $e_2\in E_{G_2}$ denote two edges with (distinct) endpoints $v_1,w_1$ and $v_2,w_2$ respectively.  Let $G=G_1 : G_2$ denote their two vertex join, which is defined by gluing together $v_1,v_2$ and $w_1,w_2$ and removing the edges $e_1,e_2$. 
\begin{figure}[h] \begin{center}
 {\includegraphics[width=10cm]{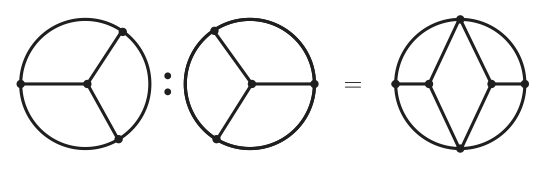}} 
\put(-220,40){$e_1$}\put(-180,40){$e_2$}
\put(-220,80){$v_1$}\put(-180,80){$v_2$}
\put(-220,5){$w_1$}\put(-180,5){$w_2$}
\end{center}
\end{figure}
Then if all integrals are finite  one may prove that 
\[ I^{\res}_G = I^{\res}_{G_1} I^{\res}_{G_2}\ .   \]  
 
\item (Planar duals). If $G$ is a graph with a planar embedding, and $G^{\vee}$ the corresponding dual graph, then one may show that 
\[ I^{\res}_G = I^{\res}_{G^{\vee}}\]  
if the corresponding integrals are finite. 
\item (Completion). Completion-invariance is a powerful constraint on Feynman residues. To state it, we assume that $G$ is  subdivergence-free and has all vertices of degree $\leq 4$.  It follows from 
\eqref{edgestwiceloops} that if $E_G>2$, then  $G$ has exactly four vertices $v_1,\ldots, v_4$ of degree three (depicted in white below).  Let $\widehat{G}$ denote the 4-regular graph obtained by adding a new vertex to $G$ and connecting it to $v_1,\ldots, v_4$.  
\begin{figure}[h] \begin{center}
 {\includegraphics[width=8cm]{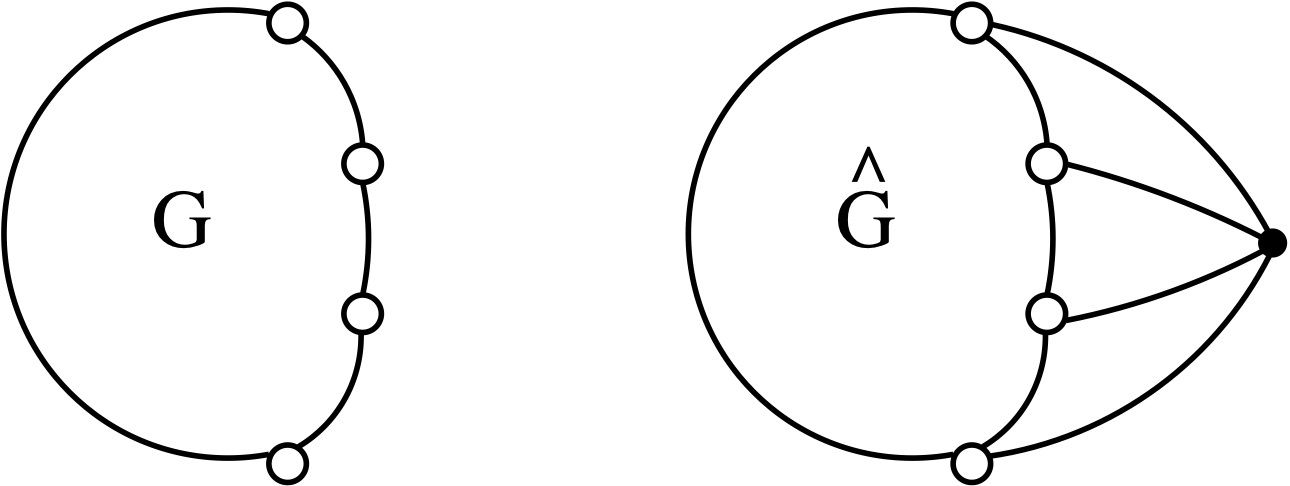}} 
\end{center}
\end{figure}

Then one may prove that the residue $I^{\res}_G$ only depends on $\widehat{G}$. In other words,  if two distinct graphs $G_1,G_2$ have isomorphic completions $\widehat{G}_1 \cong \widehat{G}_2$ then 
\[ I^{\res}_{G_1} = I^{\res}_{G_2}\ . \]
Concretely, this means that, given the completion $\widehat{G}$ of a graph as above, one may remove any vertex from it, giving rise to possibly non-isomorphic graphs whose residues are equal. 
This property of completion invariance is very far from obvious from definition \eqref{Iresdef}: indeed, the graph polynomials $\Psi_G$ of two graphs with isomorphic completions may bear little resemblance. See \cite{Schnetz} for examples. 

\item (Further identities). There exist many more known identities between Feynman residues, including various `twist' identities  to which we refer to \cite{Schnetz} for further details. They are not exhaustive: a complete combinatorial criterion for when $I^{\res}_{G_1} = I^{\res}_{G_2}$ is not known, even conjecturally, although there exist other graph invariants whose values are conjectured to coincide if and only if the corresponding Feynman residues are equal.

\end{enumerate} 

\begin{remark}
An active area of research is to construct numerical invariants of graphs which satisfy similar properties to the residue.  Examples of such invariants come from tropical geometry, number theory (counting zeros of graph polynomials over finite fields), and various combinatorial and analytic ideas (see \cite{HeppBounds} for an overview). 

In a world where artificial intelligence is playing an increasing role in the identification of patterns, it is becoming more urgent to generate data of Feynman integrals, especially in view of the applications to physics experiments.  An exact symbolic method for computing certain Feynman residues was described in \cite{BrMassless} and implemented in \cite{Panzer}. Numerical algorithms have also made great strides in recent years \cite{BorinskyTMCQ, Balduf2023}.

\end{remark}

\subsection{Contraction-Deletion}  An important property of  graph polynomials (recall \S\ref{sect: GraphPolys}), which is shared by many other invariants of graphs,  are the  contraction-deletion relations. 
For $G$ as above, let 
$ G\backslash  e $ denote the graph in which the edge $e$ is deleted and let 
\[ G /\!\! / e  = \begin{cases} 0
\quad  \hbox{ if } e \hbox{ is a self-edge} \\ 
 G/e  \quad \hbox{ otherwise}  \end{cases} \]  where $G/e$ denotes the graph in which the edge $e$ is contracted and $0$ is the `zero' graph which satisfies $\Psi_0=0$. Equivalently, $G/e$ is obtained from $G$ by deleting $e$ and identifying its endpoints in the case when they are distinct.  The contraction-deletion relations state that 
 \[  \Psi_G = x_e \Psi_{G \backslash   e}  + \Psi_{G\q e}\ .\] 
 Note that in this formula, $G \backslash  e$ may have more than one component, in which case the corresponding graph polynomial vanishes. 
  \begin{ex} Consider the following graphs obtained from the sunrise graph $G$ on the left by deleting and contracting edge $e_2$.

\begin{figure}[h]
\begin{tikzpicture}
% Draw a black circle
  \draw[black, thick] (0,0) circle (1cm);
  % Draw the horizontal diameter
  \draw[black, thick] (-1,0) -- (1,0);
  \filldraw[black] (0,0)  circle (0pt) node[anchor=south]{\small $2$};
    \filldraw[black] (0,1)  circle (0pt) node[anchor=south]{\small $1$};
      \filldraw[black] (0,-1)  circle (0pt) node[anchor=south]{\small $3$};
        \filldraw[black] (-1.5,-1)  circle (0pt) node[anchor=south]{$G$};
 % Draw the black upper semi-circle
\fill[black] (-1,0) circle (2pt);
\fill[black] (1,0) circle (2pt);
 
\draw[black, thick] (6,0) arc[start angle=0, end angle=360, radius=1cm];
% Draw the small filled black circles at endpoints
\fill[black] (6,0) circle (2pt);
\fill[black] (4,0) circle (2pt);
 \filldraw[black] (5,1)  circle (0pt) node[anchor=south]{\small $1$};
  \filldraw[black] (5,-1)  circle (0pt) node[anchor=south]{\small $3$};
   \filldraw[black] (3.5,-1)  circle (0pt) node[anchor=south]{$G\! \setminus\!  e_2$};

\fill[black] (9.5,0) circle (2pt);
\draw[black, thick] (9,0) circle (0.5cm);
\draw[black, thick] (10,0) circle (0.5cm);
 \filldraw[black] (9,-0.6)  circle (0pt) node[anchor=north]{\small $1$};
  \filldraw[black] (10,-0.6)  circle (0pt) node[anchor=north]{\small $3$};
    \filldraw[black] (8,-1)  circle (0pt) node[anchor=south]{$G/ e_2$};

\end{tikzpicture}
\end{figure}

  One has 
\[
  \Psi_G   =   \alpha_2 \Psi_{G\setminus e_2} + \Psi_{G/ e_2}   =   \alpha_2( \alpha_1 + \alpha_3) +  \alpha_1 \alpha_3 \ . \]

  \end{ex}
 The contraction-deletion relation may easily be verified from the definition, on noting that the set of spanning trees in $G$ is the disjoint union of those spanning trees which contain $e$, and those which do not. The former is in bijection with the set of spanning trees of $G\q e$, the latter with those of $G \backslash  e$.

\subsection{Graph Laplacian} Let $G$ be connected. Consider the free $\ZZ$-module $\ZZ^{E_G}$ generated by the set of edges of $G$, and equipped with the inner product defined by 
\[  \langle e_i, e_j \rangle = \delta_{ij} x_{e_i} \ . \]
Choose, in the usual fashion, an orientation on each edge and let $s(e), t(e)$ denote the source and target of $e$ respectively. The boundary map
\begin{eqnarray} \ZZ^{E_G} &  \overset{\partial}{\To} &  \ZZ^{V_G}  \nonumber  \\
e & \mapsto & t(e) - s(e) \nonumber 
\end{eqnarray} 
has kernel 
\[ H_1(G;\ZZ) := \ker \partial\]
which is a free $\ZZ$-module of rank $h_G$, generated by closed oriented cycles. 
The inner product defined above restricts to an inner product on the submodule  $H_1(G;\ZZ) \subset \ZZ^{E_G}$.

A graph Laplacian matrix\footnote{There are various different definitions of graph Laplacian matrix in the literature, the one considered here is sometimes called the dual graph Laplacian matrix.} is defined as follows.    For any basis of closed cycles $c_1,\ldots, c_{h_G}$, consider the $h_G \times h_G$ matrix $\Lambda_G$ whose entries are 
\[ (\Lambda_G)_{ij} =  \langle c_i, c_j \rangle\ . \]  
It  depends on the choice of basis, but not on the orientations of the edges of $G$. Changing the basis of $G$ modifies $\Lambda_G$ by a matrix $P \in \GL_{h_G}(\ZZ)$ via the change of basis formula $\Lambda_G \mapsto P^T \Lambda_G P$. Since $\det P \in \{1, -1\}$, it follows that the determinant of $\Lambda_G$ is well-defined. 

It follows from its definition as the restriction of an inner product that any graph Laplacian matrix $\Lambda_G$ is positive semi-definite whenever $x_e \geq 0$.

\begin{ex} Consider the graph $G$ with edges oriented as follows:

\begin{figure}[h]
\begin{tikzpicture}
\filldraw[black] (-2,0.5)  circle (0pt) node[anchor=south]{\small $1$};
\filldraw[black] (-2,-0.5)  circle (0pt) node[anchor=south]{\small $2$};
\filldraw[black] (-0.9,0)  circle (0pt) node[anchor=east]{\small $3$};
\filldraw[black] (-0.7,0)  circle (0pt) node[anchor=west]{\small $4$};
 \fill[black] (-1,1) circle(2pt);
  \fill[black] (-1,-1) circle(2pt);
   \fill[black] (-3,0) circle(2pt);
  \draw[black,thick, postaction={on each segment={mid arrow=black}}] (-1,1) .. controls (-0.6,0.4) and (-0.6, -0.4) .. (-1,-1);
  \draw[black,thick,postaction={on each segment={mid arrow=black}}] (-1,1) .. controls (-1.5,0.4) and (-1.5, -0.4)  .. (-1,-1);
  \draw[black,thick, postaction={on each segment={mid arrow=black}}] (-3,0) -- (-1,-1);
   \draw[black,thick, postaction={on each segment={mid arrow=black}}] (-3,0) -- (-1,1);
\end{tikzpicture}
\end{figure}
\noindent 
Let $c_1,c_2$ denote the cycles $c_1 = e_1+e_3-e_2$, $c_2 = e_3-e_4$, whose classes form a basis of $H_1(G;\ZZ)$.  Then, with respect to this basis, the graph Laplacian matrix is:
\[ \Lambda_G =  \begin{pmatrix} 
x_1+x_2+x_3 &  x_3 \\
x_3    & x_3+x_4 
\end{pmatrix} \ . 
\]
\end{ex}

\begin{ex}
Now consider the sunrise graph 
\begin{figure}[h]
\begin{tikzpicture}
    \draw[black, thick, postaction={on each segment={mid arrow=black}}] (0,0) circle (1cm);
  % Draw the horizontal diameter
  \draw[black, thick, postaction={on each segment={mid arrow=black}}] (1,0) -- (-1,0);
  \filldraw[black] (0,0)  circle (0pt) node[anchor=south]{\small $2$};
    \filldraw[black] (0,1)  circle (0pt) node[anchor=south]{\small $1$};
      \filldraw[black] (0,-1)  circle (0pt) node[anchor=south]{\small $3$};
     %   \filldraw[black] (-1.5,-1)  circle (0pt) node[anchor=south]{$G$};
 % Draw the black upper semi-circle
\fill[black] (-1,0) circle (2pt);
\fill[black] (1,0) circle (2pt);
\end{tikzpicture}
\end{figure}
with the cycle basis $c_1 = e_1-e_2$, $c_2 = e_2 + e_3$. The corresponding graph Laplacian matrix is 
\begin{equation} \label{LambdaSunrise}
    \Lambda_G = \begin{pmatrix}
        x_1+ x_2 & -x_2 \\
        -x_2 & x_1+x_3 
    \end{pmatrix}
\end{equation}

\end{ex} 
\begin{ex} Consider the wheel with 3 spokes with edges oriented as follows.
\begin{figure}[h]
\begin{tikzpicture}
  % Draw the outer circle
  \filldraw[black] (-0.8,0.5)  circle (0pt) node[anchor=south]{\small $1$};
   \filldraw[black] (-0.8,-1)  circle (0pt) node[anchor=south]{\small $2$};
    \filldraw[black] (0.8,0)  circle (0pt) node[anchor=south]{\small $3$};
   \filldraw[black] (-2.2,0)  circle (0pt) node[anchor=south]{\small $6$};
    \filldraw[black] (1.5,1.4)  circle (0pt) node[anchor=south]{\small $5$};
     \filldraw[black] (1.5,-1.9)  circle (0pt) node[anchor=south]{\small $4$};
  \draw[black, thick, postaction={on each segment={mid arrow=black}}] (0,0) circle(2cm);
   \fill[black,thick, ] (0,0) circle(2pt);
  % Draw the spokes
  \foreach \angle in {0, 120, 240} {
    \draw[black, thick, postaction={on each segment={mid arrow=black}}] (0,0) -- (\angle:2cm);
    \fill[black,thick] (\angle:2cm) circle(2pt); 
  }
\end{tikzpicture}
\end{figure}
With respect to the cycle basis $c_1= e_1- e_2 +e_6$, $c_2 = e_3+e_5-e_1$ , $c_3= e_2-e_3+e_4$, the graph Laplacian matrix takes the form
\[ \Lambda_{W_3} =  \begin{pmatrix} 
   x_1+x_2+x_6 & - x_1  & -x_2 \\
   - x_1 &    x_1+x_3+x_5   &  -x_3 \\
    -x_2 & -x_3 &   x_2+x_3+x_4
\end{pmatrix}
\]
\end{ex} 
The following theorem is a consequence of the `matrix-tree theorem'.
\begin{thm} The determinant  of a graph Laplacian matrix equals the graph polynomial:
\[ \det \Lambda_G = \Psi_G\ .\]
\end{thm} 
In particular, if $G$ is a connected graph then $\Lambda_G$ is positive-definite whenever $x_e>0$ for all $e\in E_G.$
The Feynman residue may  thus be expressed as a projective integral
\begin{equation} I^{\res}_G = \int_{\sigma_G}  \frac{\Omega_G}{ (\det \Lambda_G)^2} \end{equation}
for any choice of graph Laplacian $\Lambda_G$. The example \ref{exampleintrozeta3nonlinear} is the case  $G=W_3$, restricted to the affine chart $x_6=1.$

This observation shows that Feynman residues are naturally expressed in terms of determinants of matrices with linear entries, and therefore provide a canonical class of examples of the non-linear geometry introduced in the introduction. Not only do they furnish a rich and natural class of period integrals of determinantal type, but also provide a fundamental connection between number theory, combinatorics, and quantum field theory.

\subsection{General Feynman periods}
In quantum field theory one also encounters more general projective Feynman integrals of the following form
\begin{equation} \label{GeneralIG}  \int_{\sigma_G}  \frac{ N(x_e)}{\Psi_G^d}\, \Omega_G\end{equation}
for some (typically  highly structured)  numerator polynomial  $N(x_e) \in \QQ[x_e, e\in E_G] $ and for some integer $d$. The canonical integrals discussed in $\S 3$ will be  of this type. By interpreting $\Psi_G=\det \Lambda_G$  as the determinant of a graph Laplacian and restricting to an affine chart $x_e=1$, we see that  these generalised Feynman periods are all of the form \eqref{introIasdet}.

\subsection{*Physical background}  \label{sect:physics} The preceding discussion motivates the study of Feynman residues from a purely mathematical point of view.  Nevertheless, they originate  in high-energy physics where the computation of the corresponding integrals is an essential step in making predictions for particle collider experiments. 

In high-energy physics, a Feynman graph represents a certain type of interaction (called a `process') between particles according to the laws of physics as currently understood (i.e., within the framework of quantum field theory).  To every such graph is associated a  Feynman integral $I_G$, which is interpreted as a  kind of probability for that process to occur. In order to obtain a prediction for an experiment, one must sum over all possible Feynman integrals (in practice, of course, one can only compute finitely many) and take their norm squared, which gives an estimate of the probability of observing a particular scattering process.

\begin{figure}[h]
\begin{center}
\quad {\includegraphics[width=9.0cm]{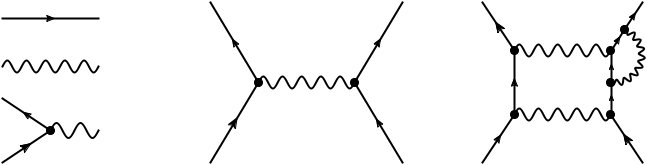}} 
\put(-160,0){$G_1$}\put(-50,0){$G_2$}
\end{center}
\caption{In quantum electrodynamics, the theory of light and matter, Feynman graphs are built out of  two types of edges and one vertex (left). The graph  $G_1$ represents, for example, two electrons (straight lines) exchanging a photon (wiggly line) which contributes to  the repulsive force experienced between two electrons. 
The graph $G_2$ represents the same process, but has two `loops', and thus represents a higher-order quantum correction to the scattering of two electrons. }
\end{figure}
Here follow some very  brief pointers of  relevance to mathematicians who  may wish to know more about this very important and extremely active part of high-energy physics: 
\begin{itemize}
    \item      A key feature of this formalism is that the graph polynomial $\Psi_G$
 is universal: it depends only on the combinatorics of the graph, whereas the choice of physical theory affects only the \emph{numerator} of the integrand \eqref{GeneralIG}. 
    This means that Feynman integrals are certain linear combinations of certain \emph{universal}  periods.

\item  In general one must consider particles with masses and momenta. To express this, one must introduce an additional graph polynomial  which is a function of this data. It has very similar properties to $\Psi_G$ considered above.   The most general Feynman integrals have both $\Psi_G$, and this  new polynomial, called the second Symanzik polynomial, in the denominator. In 4 dimensions this polynomial may  be expressed as the determinant\footnote{which also admits a `tropical' interpretation} of a square matrix of rank $h_G+1$ \cite{BrLaplacian}
Consequently,  Feynman integrals with masses and momenta are  periods of families of  motives which are universally associated to  $G$.  

\item The resulting Feynman integrals are often infinite. A very widespread technique for regularising them, called dimensional-regularisation, or `dim-reg', is to allow the parameter $d$ in  \eqref{GeneralIG} to be a real number. Divergences then appear as poles in $d$.  Consequently  \eqref{GeneralIG} is  viewed as a Mellin transform. The theory of renormalisation provides a physically-meaningful combinatorial procedure for cancelling out the divergences  coming from different graphs. 
\end{itemize}

\section{Tropical curves and their moduli}

In the previous section, we introduced an important class of periods called Feynman residues and compared and contrasted them with multiple zeta values. Their definition involved the graph Laplacian matrix, whose geometric interpretation involves the moduli space of tropical curves, which is the topic of this lecture.  We recommend Melody Chan's lecture notes \cite{MelodyTropicalLectures} for further background. 
  
\subsection{Stable weighted graphs}
Let  $G$ be a  finite  connected graph. A \emph{weighting} is a map 
\[ w:  V_G \To \ZZ_{\geq 0}\]
which assigns a non-negative integer, or \emph{weight}, to each vertex of $G$. A pair $(G,w)$ is called a weighted graph. When depicting a weighted graph, unlabelled vertices will always be assumed to have weight zero.

The \emph{genus} of a weighted graph $(G,w)$ is defined by 
\[ g(G,w) = h_G + \sum_{w \in V_G} w(v)\ . \]
A weighted graph is \emph{stable} if every vertex of weight $0$ has degree $\geq 3$ and every vertex of weight $1$ has degree $\geq 1$.

\begin{ex} There are exactly 7 stable weighted graphs of genus $2$:

\begin{figure}[h]
\begin{tikzpicture} 
\filldraw[black] (0,0) circle (2pt) node[anchor=east]{$2\,$};

\draw[black,thick] (1.5,0) circle (10pt);\filldraw[black] (1.15,0) circle (1.5pt) node[anchor=east]{$1\,$ };

\draw[black,thick] (2.8,0) -- (3.8,0); \filldraw[black] (2.8,0) circle (1.5pt) node[anchor=north]{$1$};
\filldraw[black] (3.8,0) circle (1.5pt) node[anchor=north]{$1$};

\draw[black,thick] (6,0) circle (10pt);\filldraw[black] (5.15,0) circle (1.5pt) node[anchor=east]{$1\,$ };
\draw[black,thick] (5.15,0) -- (5.65,0);\filldraw[black] (5.65,0) circle (1.5pt);

\draw[black,thick] (7.3,0) circle (10pt);
\draw[black,thick] (8,0) circle (10pt);
\filldraw[black] (7.65,0) circle (1.5pt);
\draw[black,thick] (9.7,0) circle (14pt);
\draw[black,thick] (9.2,0) -- (10.2,0);
\filldraw[black] (9.2,0) circle (1.5pt);
\filldraw[black] (10.2,0) circle (1.5pt);
\draw[black,thick] (11.3,0) circle (10pt);
\draw[black,thick] (12.5,0) circle (10pt);
\draw[black,thick] (11.65,0) -- (12.15,0);
\filldraw[black] (11.65,0) circle (1.5pt);
\filldraw[black] (12.15,0) circle (1.5pt);

%\draw[black,thick] (5.3,0) circle (10pt); \draw[black,thick] (6,0) circle (10pt); \filldraw[black] (5.65,0) circle (1.5pt); 

%\draw[black,thick] (7.7,0) circle (14pt);\draw[black,thick] (7.2,0) -- (8.2,0);
%\filldraw[black] (7.2,0) circle (1.5pt); \filldraw[black] (8.2,0) circle (1.5pt); 

%\draw[black,thick] (9.3,0) circle (10pt); \draw[black,thick] (10.5,0) circle (10pt); 
%\draw[black,thick] (9.65,0) -- (10.15,0);
%\filldraw[black] (9.65,0) circle (1.5pt); \filldraw[black] (10.15,0) circle (1.5pt); 

\end{tikzpicture}
\end{figure}

\end{ex} 

The contraction of an edge $e$ in a weighted graph $(G,w)$ is defined as follows. If $e$ has distinct endpoints $v_1,v_2$ then the graph $G/e$ has  a unique vertex $v$ which corresponds to the identification of $v_1$ and $v_2$: it is assigned the weight $w(v) = w(v_1)+ w(v_2)$. The weights of all other vertices are unchanged.

If $e$ is a self-edge based at a vertex $v$ then the contraction $(G/e,w')$ is defined to be the graph in which the edge $e$ is removed, and the vertex $v$ is assigned the weight $w'(v) = w(v)+1$.  The weighting $w'$ is defined to be equal to $w$ for all other vertices.

\begin{figure}[h]
\begin{tikzpicture} 
\draw[black,thick] (0,2) -- (2,2) node[midway, above] {$e$}; ; \filldraw[black] (0,2) circle (1.5pt) node[anchor=north]{$w_1$};
\filldraw[black] (2,2) circle (1.5pt) node[anchor=north]{$w_2$};
\filldraw[black] (0,2) circle (1.5pt); 
\draw[thick, decorate,decoration={snake,amplitude=1mm,segment length=5mm},->] (3,2) -- (4,2)  node[midway, above] {$/e$};
\filldraw[black] (5,2) circle (1.5pt) node[anchor=north]{$w_1\! +\! w_2$};

\draw[black,thick] (1,0.35) circle (10pt); \node at (1, 0.9)  {$e$};; 
\filldraw[black] (1,0) circle (1.5pt) node[anchor=north]{$w$};
\draw[thick, decorate,decoration={snake,amplitude=1mm,segment length=5mm},->] (3,0.35) -- (4,0.35)  node[midway, above] {$/e$};
\filldraw[black] (5,0.35) circle (1.5pt) node[anchor=north]{$w\! +\! 1$};

\end{tikzpicture}
    
\end{figure}
With this definition, we see that the role of the weighting is to keep track of contracted self-edges (and by extension, loops), and to ensure that the genus $g(G,w)$ is preserved under the contraction of all  edges, whether they be self-edges or not.

\subsection{*Geometric interpretation} 
The connected stable weighted graphs are in one-to-one correspondence with the boundary strata of the Deligne-Mumford compactification $\overline{\mathcal{M}}_g$ of the moduli space $\mathcal{M}_g$ of Riemann surfaces of genus $g$, for $g\geq 2$. 

In this correspondence, vertices of weight $g$ correspond to irreducible curves of genus $g$, and edges correspond to nodes. The contraction of edges are dual to nodal  degenerations: the latter may be visualised  by tying a rope around a closed curve on a surface of genus $g$ and pulling it tight until it pinches to a nodal point. See below for a picture of genus 2.

\begin{figure}[h] \begin{center}
 {\includegraphics[width=14cm]{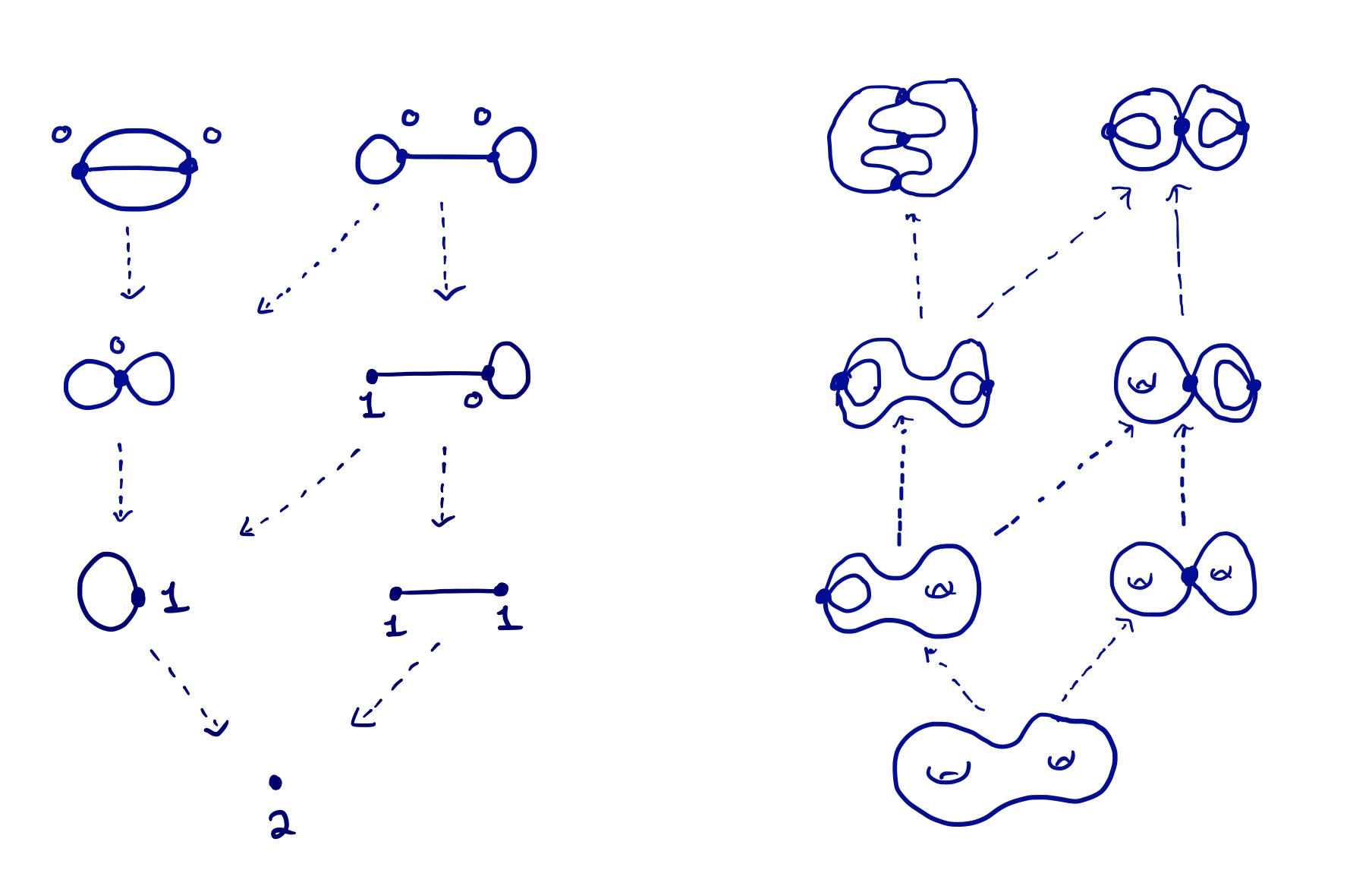}} 
\end{center}
\end{figure}

Since the interpretation of boundary strata of $\overline{\mathcal{M}}_g$ is very well-known  we shall say no more about it here.  However, it is important to bear in mind, when comparing with the  moduli space of tropical curves, that the  correspondence is dimension-reversing (for example, a single vertex of weight $g$ corresponds to the big open stratum $\mathcal{M}_g$ of $\overline{\mathcal{M}}_g$,  but will correspond to the cone point in the tropical moduli space which is discussed below). 

\subsection{Metric graphs}
A \emph{metric} on a connected finite graph $G$ is a map
\[ \ell:  E_G \To \RR_{>0}\]
which to each edge of $G$ assigns a positive length  $\ell_e>0$. 
A \emph{tropical curve}, for our purposes, is nothing other than a \emph{weighted, stable, metric graph}.
Often one refers to the  `combinatorial type' of a tropical curve as the stable, weighted graph which is obtained  from it  by forgetting all the edge lengths.

\subsection{Moduli space of metric graphs} The moduli space of metrics on a given combinatorial type $G$ is given by the open quadrant
\[ C_G =  \RR_{>0}^{E_G}  \ .  \]
 We denote its closure in $  \RR^{E_G}$ by 
 \[ \overline{C}_G = \RR_{\geq 0}^{E_G}  = \{ (\ell_e)_{e\in E_G}: \ell_e \geq 0 \}\ . \]
Note that if $E_G$ is empty, then $\overline{C}_G$ is the empty product,   and hence reduces to a point. 
 Points in $\overline{C}_G$ correspond to metric graphs with combinatorial type $G$ in which we also allow edge lengths $\ell_e$ to be zero. The moduli space of tropical curves of genus $g$ is defined to be 
 \[ \mathcal{M}_g^{\trop} =  \left( \bigcup_{G} \overline{C}_G \right) / \sim \] 
where the union runs over all stable, weighted graphs of genus $g$, and the equivalence relation $\sim$ is generated by: 
\begin{itemize}
\item Isomorphisms  for all $G, e\in E_G$ of the form
\[ \overline{C}_{G/e}  \overset{\sim}{\To} \{\ell_e=0\} \subset  \overline{C}_G\]
which identifies a metric graph $G/e$ with the  metric graph $G$ for which the length $\ell_e$ is zero and all other edges have the same lengths. 
\item  The action of the group of automorphisms of $G$ on $\overline{C}_G$ by permuting the coordinates of $\RR_{\geq 0}^{E_G}$. More precisely, 
$\sigma: (\ell_e)_{e\in E(G)} \mapsto (\ell_{\sigma(e)})_{e\in E(G)}$ for all $\sigma \in \mathrm{Aut}(G)$. 
\end{itemize}
Although the cube $\overline{C}_G$ of edge-lengths does not depend on the weighting $w$ of $G$, the group of automorphisms $\mathrm{Aut}(G) = \mathrm{Aut}(G,w)$ does.
Note also that only the image of the automorphism group  $\mathrm{Aut}(G) 
\rightarrow \mathrm{Aut}'(G) \leq \Sigma_{E(G)}$ of $G$ in the group of permutations of the set of edges of $G$ plays a role in the above definition. Thus  here we are considering a `coarse' version of the moduli space of graphs. See \cite{ModuliGraphsStack}  for an orbifold version in which quotients are taken with respect to the full group of automorphisms of $G$.

\begin{ex} Consider the following weighted stable graph of genus $3$.

\begin{figure}[h]\begin{tikzpicture}
  \begin{scope}[shift={(-1cm,0)}]
    \node at (0.5,-1.5)  {$\mathrm{Aut}' (G )\cong \mathbb{Z}/2\mathbb{Z}$};
    \draw[blue,thick] (0,0) .. controls (0.333,0.4) and (0.666,0.4) .. (1,0);
    \node at (0.5,0.6)  {$\ell_1$};
    \fill(0.5,-0.3) circle (0pt)  node[below] {$\ell_2$};
    \draw[blue,thick] (0,0) .. controls (0.333,-0.4) and (0.666,-0.4) .. (1,0);
    \filldraw[blue,thick] (0,0) circle (1.5pt) node[left] {\small{$1$}};
    \filldraw[blue,thick] (1,0) circle (1.5pt) node[right] {\small{$1$}};
  \end{scope}

  \draw[blue,thick,->] (3.5,-2) -- (7,-2);
 \draw[blue,thick,->] (4,-2.5) -- (4,1);
\draw[red, thick, dashed] (4,-2) -- (7,1);

\draw[red, thick, <->] (7,-1+1) .. controls (6.8,-0.6+1) and (6.6,-0.4+1) .. (6,0+1);

 \filldraw[blue] (3.6,-2.3) circle (1.5pt) node[anchor=north]{$3$};
  \filldraw[blue] (5.5,-2.3) circle (1.5pt) node[anchor=east]{$2$};
  \draw[blue,thick] (5.7,-2.3) circle (6pt);   \node at (6.1,0.-2.3) {$\ell_1$}; 
   \draw[blue,thick] (3.4,-0.4) circle (6pt);   \node at (3.8,-0.4) {$\ell_2$}; 
    \filldraw[blue] (3.2,-0.4) circle (1.5pt) node[anchor=east]{$2$};

\draw[blue,thick] (4.8,0.2) .. controls (4.9665,0.4) and (5.133,0.4) .. (5.3,0.2); \node at (5.05,0.55) {$\ell_1$}; \fill (5.05,0.06) circle (0pt) node[below] {$\ell_2$}; \draw[blue,thick] (4.8,0.2) .. controls (4.9665,0) and (5.133,0) .. (5.3,0.2); \filldraw[blue,thick] (4.8,0.2) circle (1.5pt) node[left] {\small{$1$}}; \filldraw[blue,thick] (5.3,0.2) circle (1.5pt) node[right] {\small{$1$}};

\end{tikzpicture}
\end{figure}
\noindent 
Denote its edge lengths by $\ell_1,\ell_2 > 0$.  Its automorphism group is $\mathrm{Aut}(G) = \ZZ/2\ZZ \times \ZZ/2\ZZ$; the quotient acting as edge permutations  is $\mathrm{Aut}'(G) \cong \ZZ/2\ZZ. $  On the right is the corresponding  cube $\overline{C}_G$ with coordinates $(\ell_1,\ell_2)$ where $\ell_1,\ell_2 \geq 0$. The edges and corner of this cube  are identified with the corresponding cubes of graphs where edges $1,2$ are contracted. The automorphism group acts on $\overline{C}_G$ by interchanging $\ell_1 \leftrightarrow \ell_2$ and corresponds to a symmetry along the  dashed red diagonal line. 
\end{ex} 

The topological space $\mathcal{M}_g^{\trop}$ is a cone with cone point indexed by the graph consisting of a single vertex of weight $g$. As a consequence, it is contractible and its topology is uninteresting. For this reason we prefer to work with the associated link.
\subsection{The space $\LM_g^{\trop}$} The link is defined to be 
\[ \LM_g^{\trop}  = \left( \mathcal{M}_g^{\trop} \setminus \hbox{cone point}\right) / \RR^{\times}_{>0} \ .   \]  
It may be described directly by gluing together simplices rather than cones. Let 
\[ \sigma_G = \{ (\ell_e)_{e\in E_G} \in \RR_{>0}^{E_G} \hbox{ such that } \sum_{e \in E_G} \ell_e = 1\  \} \ .  \] 
It consists of all possible non-zero edge lengths on $G$ which are normalised so that they sum to $1$.
Denote its closure by 
\[ \overline{\sigma}_G =    \{  \ell_e \geq 0  \hbox{ such that } \sum_{e \in E_G} \ell_e = 1\} \ .  \] 
We have
\[ \LM_g^{\trop} =  \left( \bigcup_{G} \sigma_G \right) / \sim \] 
where the relation $\sim$ is as before: namely we identify metric  graphs with an edge $e$ of length $0$ with the corresponding quotient $G/e$, and take the  quotient by automorphisms. 

\begin{ex} We may give a complete description of $\LM_2^{\trop}$ as follows.

Consider first the sunrise diagram $G$ with edge lengths $(\ell_1,\ell_2,\ell_3)$.  Its cell $\overline{\sigma}_G$ is the  closed 2-simplex where $\ell_1+\ell_2+\ell_3=1$, and is depicted below: 
\begin{figure}[h]
\begin{tikzpicture} 
\draw[blue, ultra thick] (0,0) -- (4,0) -- (2,3.5) -- cycle;
\filldraw[blue] (0,0) circle (3pt);\filldraw[blue] (4,0) circle (3pt);\filldraw[blue] (2,3.5) circle (3pt);

\draw[red,thick] (-0.2,-0.4) circle (6pt);\filldraw[red] (-0.4,-0.4) circle (1.5pt) node[anchor=east]{\tiny{$1$}};
\node at (0.2,-0.4) {\tiny{$\ell_3$}} ; 
\node at (1.4,-0.4) {\tiny{$\ell_2$}} ; 
\node at (2.6,-0.4) {\tiny{$\ell_3$}} ; 
\node at (4.8,-0.4) {\tiny{$\ell_2$}} ;

\node at (0.2,2.4) {\tiny{$\ell_1$}} ; \node at (0.8,2.4) {\tiny{$\ell_3$}} ; 
\node at (3.2,2.4) {\tiny{$\ell_1$}} ; \node at (3.8,2.4) {\tiny{$\ell_2$}} ; 
\node at (2,4.4) {\tiny{$\ell_1$}} ;
\draw[red,thick] (4.8-0.4,-0.4) circle (6pt);\filldraw[red] (4.8-0.6,-0.4) circle (1.5pt) node[anchor=east]{\tiny{$1$}};

\draw[red,thick] (2.5-0.4,4.4-0.4) circle (6pt);\filldraw[red] (2.5-0.6,4.4-0.4) circle (1.5pt) node[anchor=east]{\tiny{$1$}};

%double rose
\draw[red,thick] (2.8-1,-0.4) circle (6pt);\draw[red,thick] (2.8-0.6,-0.4) circle (6pt);\filldraw[red] (2.8-0.8,-0.4) circle (1.5pt);

\draw[red,thick] (1.3-1,2) circle (6pt);\draw[red,thick] (1.3-0.6,2) circle (6pt);\filldraw[red] (1.3-0.8,2) circle (1.5pt);

\draw[red,thick] (4.3-1,2) circle (6pt);\draw[red,thick] (4.3-0.6,2) circle (6pt);\filldraw[red] (4.3-0.8,2) circle (1.5pt);

%sunrise
\draw[red,thick] (2,1.4) circle (8pt);
\draw[red, thick] (2-0.3,1.4) -- (2+0.3,1.4);
\filldraw[red] (2-0.275,1.4) circle (1.5pt);\filldraw[red] (2+0.275,1.4) circle (1.5pt);

  \coordinate (A) at (0+6,0);
  \coordinate (B) at (4+6,0); % reduced from 6 to 2/3
  \coordinate (C) at (2+6,3.464); % reduced from 3 & 5.196 to 2/3
  
  % Draw the thick blue equilateral triangle
  \draw[ultra thick, blue] (A) -- (B) -- (C) -- cycle;
  
  % Midpoints for barycentric subdivision
  \coordinate (D) at ($(A)!0.5!(B)$);
  \coordinate (E) at ($(B)!0.5!(C)$);
  \coordinate (F) at ($(C)!0.5!(A)$);

  % Draw red dashed lines for the barycentric subdivision
  \draw[blue, dashed] (A) -- (E);
  \draw[blue, dashed] (B) -- (F);
  \draw[blue, dashed] (C) -- (D);

\end{tikzpicture}
\end{figure}

\noindent 
Its automorphism group is  $\Sigma_2\times \Sigma_3$, the group $\mathrm{Aut}'(G) = \Sigma_3$  acts on $\overline{\sigma}_G$ by permuting the coordinates $(\ell_1,\ell_2,\ell_3)$.  The quotient  $\overline{\sigma}_G/ \Sigma_3$  of this simplex by the action of the symmetric group is depicted on the right: a fundamental region for the action of $\Sigma_3$ is given by any of the 6 small triangles in the subdivison.

Now consider the cell $\overline{\sigma}_G$ where $G$ is the dumbbell diagram, whose automorphism group is of order $8$. 
\begin{figure}[h!]
\begin{tikzpicture}
\draw[blue, ultra thick] (7,0) -- (11,0) -- (9,3.5) -- cycle;
\filldraw[blue] (7,0) circle (3pt);\filldraw[blue] (11,0) circle (3pt);\filldraw[blue] (9,3.5) circle (3pt);

\draw[red,thick] (7.2,2) circle (6pt);
\draw[red, thick] (7.2+0.2,2) -- (7.2+0.5,2);
\filldraw[red] (7.2+0.2,2) circle (1.5pt);\filldraw[red] (7.2+0.5,2) circle (1.5pt) node[anchor=north]{\tiny{$1$}};

\node at (6.5,-0.4) {\tiny{$\ell_1$}} ; 
\node at (8.4,-0.4) {\tiny{$\ell_1$}} ; \node at (9.6,-0.4) {\tiny{$\ell_3$}} ; 
\node at (11.7,-0.4) {\tiny{$\ell_3$}} ;

\node at (7.2,2.4) {\tiny{$\ell_1$}} ; \node at (7.6,2.2) {\tiny{$\ell_2$}} ; 
\node at (10.7,2.4) {\tiny{$\ell_3$}} ; \node at (10.3,2.2) {\tiny{$\ell_2$}} ;

\node at (9,4.2) {\tiny{$\ell_2$}} ;

\node at (9,1.6) {\tiny{$\ell_2$}} ; 
\node at (8.4,1.8) {\tiny{$\ell_1$}} ; 
\node at (9.6,1.8) {\tiny{$\ell_3$}} ; 

\draw[red,thick] (11.7-0.4,-0.4) circle (6pt);\filldraw[red] (11.7-0.6,-0.4) circle (1.5pt) node[anchor=east]{\tiny{$1$}};

\draw[red,thick] (10+0.7,2) circle (6pt);
\draw[red, thick] (10+0.2,2) -- (10+0.5,2);
\filldraw[red] (10+0.2,2) circle (1.5pt) node[anchor=north]{\tiny{$1$}};\filldraw[red] (10+0.5,2) circle (1.5pt) ;

\draw[red,thick] (7.7-0.8,-0.4) circle (6pt);\filldraw[red] (7.7-0.6,-0.4) circle (1.5pt) node[anchor=west]{\tiny{$1$}};

\draw[red,thick] (9.8-1,-0.4) circle (6pt);\draw[red,thick] (9.8-0.6,-0.4) circle (6pt);\filldraw[red] (9.8-0.8,-0.4) circle (1.5pt);

\draw[red, thick] (8.75,4) -- (9.25,4);
\filldraw[red] (9.25,4) circle (1.5pt) node[anchor=west]{\tiny{$1$}};
\filldraw[red] (8.75,4) circle (1.5pt) node[anchor=east]{\tiny{$1$}};

%dumbell
\draw[red, thick] (8.7,1.4) -- (9.3,1.4);
\filldraw[red] (8.7,1.4) circle (1.5pt);\filldraw[red] (9.3,1.4) circle (1.5pt);
\draw[red,thick] (8.5,1.4) circle (6pt);
\draw[red,thick] (9.5,1.4) circle (6pt);

  \coordinate (A) at (0+14,0);
  \coordinate (B) at (4+14,0); % reduced from 6 to 2/3
  \coordinate (C) at (2+14,3.464); % reduced from 3 & 5.196 to 2/3
  
  % Draw the thick blue equilateral triangle
  \draw[ultra thick, blue] (A) -- (B) -- (C) -- cycle;
  
  % Midpoints for barycentric subdivision
  \coordinate (D) at ($(A)!0.5!(B)$);
  \coordinate (E) at ($(B)!0.5!(C)$);
  \coordinate (F) at ($(C)!0.5!(A)$);

  % Draw red dashed lines for the barycentric subdivision
%  \draw[blue, dashed] (A) -- (E);
 % \draw[blue, dashed] (B) -- (F);
  \draw[blue, dashed] (C) -- (D);
\end{tikzpicture}
  
\end{figure}
The image of the automorphism group  $\mathrm{Aut}'(G)$  is of order $2$, and it acts on $\overline{\sigma}_{G}$ by interchanging the coordinates $\ell_1$ and $\ell_3$.

The space $\LM_2^{\trop}$ is obtained by gluing the quotients of the simplices associated to these two graphs by their automorphism groups, along all edges and vertices which correspond to isomorphic metric graphs (which have equivalent labels in the figures shown). 
\end{ex} 
What kind of  objects are the spaces $\mathcal{M}_g^{\trop}$ and $\LM_g^{\trop}$?  They are \emph{a priori}  topological spaces but carry additional structure since they are locally built out of quotients of Euclidean spaces (respectively,  simplices).  Note that they  are not   manifolds since they may locally  resemble the pages of a book,  joined  together along the spine (see \cite{WhatIsOuterSpace}). They also have an `orbifold' nature along the fixed points for the action of  $\mathrm{Aut}(G)$ on $\sigma_G$.  In early references, such objects were called `stacky fans', but  the precise category in which one chooses to view them, and the corresponding terminology, continues to evolve.
 
\subsection{Open weight zero locus, and boundary}

\begin{defn}
Define  the \emph{boundary}  of $\LM_g^{\trop}$ to be the space: 
\[  \partial \LM_g^{\trop} = \left(\bigcup_{G, w(G)>0}  \overline{\sigma}_G \right) / \sim \]
where the sum is over all stable, weighted  graphs  (combinatorial types) of genus $g$ and total positive weight, where the total weight of a graph $G$ is defined by:
\[ w(G) = \sum_{v\in V_G} w(v) \ .\]
In other words, $ \partial \LM_g^{\trop} \subset   \LM_g^{\trop} $ is the subspace indexed by graphs which have at least one vertex of weight $>0$.
 
 Define the weight zero locus to be the open complement
 \[   \LM_g^{\circ, \trop}  = \LM_g^{\trop} \setminus  \partial \LM_g^{\trop} \ .\]
 It is the open locus $ \LM_g^{\circ, \trop}  \subset \LM_g^{\trop} $ consisting of the images of the \emph{open} simplices $\sigma_G$ for all graphs $G$ which have total weight $w(G)=0$.
 \end{defn}

The open locus $\LM_g^{\circ, \trop}$ also goes by the name of the quotient of \emph{outer space} \cite{CullerVogtmann}  $\mathcal{O}_g$ by the action  of the outer automorphisms of the free group on $g$ generators:
\[   \LM_g^{\circ, \trop} = \mathcal{O}_g/ \mathrm{Out}(F_g) \ .\]
It is shown in \cite{CGP} that the boundary $\partial \LM_g^{\trop}$ is contractible.

\begin{figure}[h]
\begin{tikzpicture} 
\draw[blue, ultra thick] (0,0) -- (4,0) -- (2,3.5) -- cycle;
\filldraw[white] (0,0) circle (3pt);\filldraw[white] (4,0) circle (3pt);\filldraw[white] (2,3.5) circle (3pt);

\node at (1.4,-0.4) {\tiny{$\ell_2$}} ; 
\node at (2.6,-0.4) {\tiny{$\ell_3$}} ;

\node at (0.2,2.4) {\tiny{$\ell_1$}} ; \node at (0.8,2.4) {\tiny{$\ell_3$}} ; 
\node at (3.2,2.4) {\tiny{$\ell_1$}} ; \node at (3.8,2.4) {\tiny{$\ell_2$}} ;

%double rose
\draw[red,thick] (2.8-1,-0.4) circle (6pt);\draw[red,thick] (2.8-0.6,-0.4) circle (6pt);\filldraw[red] (2.8-0.8,-0.4) circle (1.5pt);

\draw[red,thick] (1.3-1,2) circle (6pt);\draw[red,thick] (1.3-0.6,2) circle (6pt);\filldraw[red] (1.3-0.8,2) circle (1.5pt);

\draw[red,thick] (4.3-1,2) circle (6pt);\draw[red,thick] (4.3-0.6,2) circle (6pt);\filldraw[red] (4.3-0.8,2) circle (1.5pt);

%sunrise
\draw[red,thick] (2,1.4) circle (8pt);
\draw[red, thick] (2-0.3,1.4) -- (2+0.3,1.4);
\filldraw[red] (2-0.275,1.4) circle (1.5pt);\filldraw[red] (2+0.275,1.4) circle (1.5pt);

\end{tikzpicture}
\caption{The part of the  open locus $\LM_2^{\circ,\trop}$ corresponding to the sunrise graph is the quotient of the above figure, which depicts a closed triangle minus its three corners, by the action of the symmetric group $\Sigma_3$. The three corners of the triangle map to a single point in $\partial \LM_2^{\trop}$.  } \label{fig: OPENsunrisesimplex}
\end{figure}
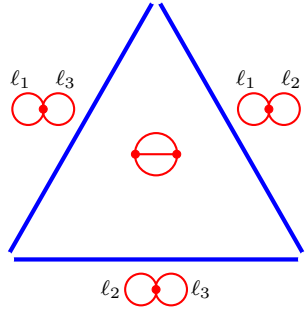

\subsection{*Graph polynomial vanishing locus} The space $\LM_g^{\trop}$ is closely related to the geometry of Feynman integrals.
First of all, observe that there is a canonical isomorphism:
\[  \sigma_G \cong \{ (x_e)_{e \in E_G} \subset \PP^{E_G}(\RR): x_e >0 \}\]
 between the  simplices $\sigma_G$ whose points are metric graphs with normalised edge lengths defined in this lecture, and the coordinate simplices in projective space
considered in the first lecture. The isomorphism is given by $\ell_e = x_e/(\sum_{e\in E_G} x_e)$ and  extends to closed simplices.  This  shows that the domain of Feynman integration may be naturally identified with the moduli space of metric graphs of fixed combinatorial type.

For any  unweighted metric graph $G$   (i.e., all vertex weights are zero) we shall denote by $\Lambda_G$ a choice of graph Laplacian matrix in which the edge variables $x_e$ are replaced with the lengths $\ell_e$ of each edge $e\in E_G$. Thus, for an unweighted metric graph, $\Lambda_G$ is a symmetric matrix with \emph{real entries}. Its determinant $\Psi_G= \det \Lambda_G$ is a real number.    

\begin{lem} The graph polynomial  $\Psi_G$  vanishes precisely along the part of the boundary of $\overline{\sigma}_G$  in 
  $\PP^{E_G}(\RR)$ which consists of  faces  $\overline{\sigma}_{G'}$   indexed by graphs   $G'$ of positive weight, where $G'$ is  obtained from $G$ by contracting edges.
\end{lem} 

\begin{proof} By definition 
\[ \Psi_G = \sum_T \prod_{e\notin T} x_e\]
 is a sum of monomials  with positive coefficients.  Every connected graph has at least one spanning tree, and so $\Psi_G\neq 0$. Since $x_e\geq 0$ on the simplex $\overline{\sigma}_G$, it follows that $\Psi_G\geq0$ on $\overline{\sigma}_G$ and vanishes if and only if every monomial vanishes, which can only happen if at least one $x_e$ vanishes.  By contraction-deletion,  $\Psi_G\big|_{x_e=0} = \Psi_{G\q e}$ and we may repeat the argument with $G$ replaced by $G\q e$ if the latter is non-empty, which happens if and only if $e$ is not a self-edge and hence $G\q e= G/e$ has  total weight zero.  We have shown that $\Psi_G>0$ on the locus consisting of graphs of zero total weight.   Conversely, by contraction-deletion again, one sees that $\Psi_G$ vanishes precisely along those faces which  are obtained by  contracting loops, or equivalently, those  which have positive total weight. 
\end{proof}

Note that both the simplices and graph polynomial have the contraction property:
\begin{eqnarray}
\overline{\sigma}_G \cap \{x_e=0\}  & = &  \overline{\sigma}_{G/e}  \nonumber \\
 \Psi_G \Big|_{x_e=0}  & = &  \Psi_{G\q e}  \nonumber
\end{eqnarray} 
\begin{ex} Consider the  sunrise graph of example \ref{ex: sunrise}.  The vanishing locus of $\Psi_G$ is  the quadric in projective space 
\[Q=\{ (x_1:x_2:x_3) \in \PP^2  :  x_1x_2+x_1x_3+x_2x_3=0\} \ . \]
It meets the closed coordinate simplex $\overline{\sigma}_G$ in exactly 3 points, namely the three vertices $(1:0:0)$, $(0:1:0)$, $(0:0:1)$  of the triangle in figure \ref{fig3}. These  correspond to setting two out of three of the edge lengths to zero (i.e., contracting two edges). 
\end{ex}

\begin{figure}[h]
\begin{tikzpicture}
 \draw[blue,thick] (-2.2,-2) -- (0.3,2.5);
  \draw[blue,thick] (-0.3,2.5) -- (2.2,-2);
   \draw[blue,thick] (-2.4,-1.1) -- (2.4,-1.1);
 \draw[red, very thick] (0,-0.05) circle (2cm); 
 \node[red] at (2.2,1.5) {$\Psi_G=0$};
  \node[blue] at (-3.1,-1.15) {$x_1=0$};
  \node[blue] at (1,2.4) {$x_2=0$};\node[blue] at (2.9,-1.9) {$x_3=0$};
  \node at (0,0) {$\sigma_G$};
\end{tikzpicture}
\caption{A picture of the geometry associated to the sunrise graph $G$. The graph hypersurface $V(\Psi_G)$ in $\mathbb{P}^2$ is depicted in red; the coordinate hyperplanes in blue. The domain of integration $\sigma_G$ is the coordinate simplex which meets the graph hypersurface in the three corners. By identifying the projective simplex $\{(x_1: x_2: x_3) : x_i \geq 0\}$ with  the region in real space $\{(\ell_1,\ell_2,\ell_3): \sum_{i=1}^3 \ell_i =1\}$ , one sees that the relative closure of the domain of integration in $\mathbb{P}^2 \setminus V(\Psi_G)$ may be identified with the region of figure \ref{fig: OPENsunrisesimplex}  (before taking the quotient by graph automorphisms).   }\label{fig3} 
\end{figure}
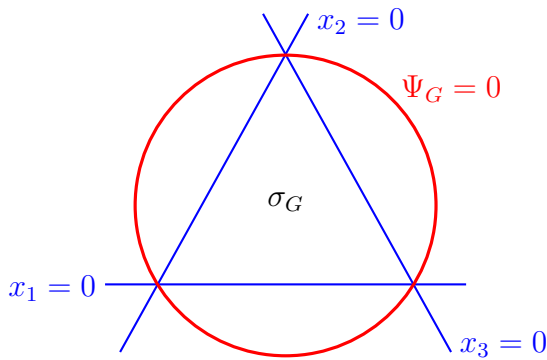

It follows that the open locus $\LM_g^{\circ, \trop}$ is precisely the region where the graph polynomial is strictly positive, while its  boundary $\partial \LM_g^{\trop}$  corresponds precisely to the vanishing locus of $\Psi_G$  (although the graph polynomial itself is only defined  up to scalar multiplication by $\RR_{>0}^{\times}$ on the link $\LM_g^{\trop}$, its (non)-vanishing locus makes sense). Thus   the boundary strata correspond geometrically to degenerations of metric graphs obtained by contracting edges, and analytically to  the regions where the integrand of Feynman integrals (and canonical integrals to be discussed below) develop singularities.

\subsection{**Description as an inductive limit}
The moduli space of tropical curves, and its link, have a very simple description as an inductive limit of topological spaces. 

Let $I_g$ denote the category whose objects are isomorphism classes of stable weighted (combinatorial) graphs of genus $g$ with at least one edge, with morphisms generated by isomorphisms, and contractions of edges. Let $I_g^{\mathrm{opp}}$ denote the category with the same objects, but all morphisms reversed. There is a functor from 
$I_g^{\mathrm{opp}}$ to topological spaces which sends a stable weighted graph $G$ to the closed simplex $\overline{\sigma}_G$. Then 
\[ \LM_g^{\trop} = \varinjlim_{G \in I_g^{\mathrm{opp}}} \overline{\sigma}_G \ .  \]

\section{Graph complex and Grothendieck-Teichm\"uller Lie algebra}

In this lecture, we shall discuss the commutative, even graph complex.  It provides a combinatorial model for the topology of the moduli space of tropical curves, and will allow us to relate the geometry developed in the previous section to the algebraic structures governing multiple zeta values.

We first discuss graph homology and  then  review its relationship to the Grothendieck-Teichm\"uller Lie algebra and the structure of the ring of  MZV's.

\subsection{Graph homology}
Let $G$ be a finite connected graph. In the literature, it is usual to assume that $G$ contains no self-edges and has no vertices of degree $\leq 2$. There are many variants,  which in most cases lead to more or less the same graph homology. 

An \emph{orientation} on  a graph $G$ in this context\footnote{not to be confused with the earlier notion of an orientation on each edge of $G$} is the sign of an ordering on the set of edges of $G$. More precisely, it is an element
\[  \varpi = e_1 \wedge \ldots \wedge e_{n}  \quad  \in \quad  \left( \textstyle{\bigwedge}^n  \, \ZZ^{E_G}\right)^{\times}\hookrightarrow \textstyle{\bigwedge}^n \,\ZZ^{E_G} \]
where the edges $E_G$ of $G$ are ordered from $1$ to $n=|E_G|$. Interchanging  two edges in this edge-ordering results in the opposite sign. An oriented graph is the data $(G,\varpi)$. 

Let $\GC_2$ denote the $\QQ$-vector space spanned by pairs $(G,\varpi)$ modulo the relations:
\begin{itemize}
\item $(G, \varpi)  =   - (G, -\varpi) $ 
\item For any isomorphism $\tau: G \cong G'$, we have 
\[  (G, \varpi) =  (G', \tau \varpi)\]
where $\tau \varpi$ is the orientation on $G'$ induced by $\tau: E_G \cong E_{G'}$.  
\end{itemize} 
Let us denote the equivalence class of $(G,\varpi)$ by $[G,\varpi]$. 
For every integer $N$, there exists a graph complex denoted by $\GC_N$.  Only the parity of $N$ significantly changes the complex, and thus one speaks of the `odd' or `even' graph complex depending on whether $N$ is odd or even. The odd graph complex will not be considered here but looks very different from the even one since the corresponding notion of orientation is quite different.  

When drawing examples, we will typically not specify the orientation: some choice of orientation will be implicitly assumed.

\begin{ex} Consider any graph which contains a bubble (two edges $1,2$ which share the  same two endpoints). Then by interchanging edges $1,2$
we have:
    
\begin{figure}[h]\begin{tikzpicture}
 
  % Original picture
  \draw[blue,thick] (0,0) .. controls (0.333,0.4) and (0.666,0.4) .. (1,0);
  \draw[blue,thick] (0,0) .. controls (0.333,-0.4) and (0.666,-0.4) .. (1,0);
  \filldraw[blue] (0,0) circle (1.5pt);
  \filldraw[blue] (1,0) circle (1.5pt);
  \draw[blue,thick] (0,0) -- (-0.1,-0.1);
  \draw[blue,thick] (0,0) -- (-0.1,0);
  \draw[blue,thick] (0,0) -- (-0.1,0.1);
  \draw[blue,thick] (1,0) -- (1.1,-0.1);
  \draw[blue,thick] (1,0) -- (1.1,0);
  \draw[blue,thick] (1,0) -- (1.1,0.1);
  \node at (0.5,0.5) {{\tiny $1$}};
  \node at (0.5,-0.5) {{\tiny $2$}};
  \node at (2.3,0) {$\, ,\, e_1 \wedge e_2 \wedge \omega\Big)$};
  \node at (-0.3,0) {$\Big($};

  % Duplicated picture moved  to the right
  \begin{scope}[shift={(4.7,0)}]
    \draw[blue,thick] (0,0) .. controls (0.333,0.4) and (0.666,0.4) .. (1,0);
    \draw[blue,thick] (0,0) .. controls (0.333,-0.4) and (0.666,-0.4) .. (1,0);
    \filldraw[blue] (0,0) circle (1.5pt);
    \filldraw[blue] (1,0) circle (1.5pt);
    \draw[blue,thick] (0,0) -- (-0.1,-0.1);
    \draw[blue,thick] (0,0) -- (-0.1,0);
    \draw[blue,thick] (0,0) -- (-0.1,0.1);
    \draw[blue,thick] (1,0) -- (1.1,-0.1);
    \draw[blue,thick] (1,0) -- (1.1,0);
    \draw[blue,thick] (1,0) -- (1.1,0.1);
    \node at (0.5,0.5) {{\tiny $2$}};
    \node at (0.5,-0.5) {{\tiny $1$}};
    \node at (2.3,0) {$\, ,  \, e_2 \wedge e_1 \wedge \omega\Big)$};
    \node at (-0.6,0) {$=\,\, \,  \, \Big($};
  \end{scope}

    \end{tikzpicture}
    \end{figure}
    \noindent 
Since $e_1 \wedge e_2= - e_2 \wedge e_1$, it follows that  
$ {[}G,\varpi]=0$ 
whenever $G$ has multiple edges, since there exists a symmetry which permutes any two edges with common endpoints.
\end{ex}

The vector space $\GC_2$ is equipped with a differential. Consider the linear map defined on generators by the formula
\begin{eqnarray} 
d: \GC_2 & \To & \GC_2 \nonumber \\ 
d {[} G, e_1\wedge \ldots \wedge e_n] & = & \sum_{i=1}^n (-1)^i  {[} G\q e_i , e_1 \wedge \ldots \wedge \widehat{e_i} \wedge \ldots \wedge e_n{]} 
 ,   \nonumber
\end{eqnarray} 
which contracts each edge of $G$ in turn. One may check that the map $d$ is well-defined and satisfies $d^2=0$, which follows from the sign rule in the previous formula together with the fact that contraction of edges commute: the graph obtained by contracting $e_i$ followed by $e_j$ in $G$ is isomorphic to the one where $e_j$  is contracted before  $e_i$. 

\begin{ex}
The contraction of any edge which lies in a triangle is zero. Consider the differential on any oriented graph which contains a triangle, as shown below.  

\begin{figure}[h]
\begin{tikzpicture}
 \node at (-0.35,0.3) { $d \Big[$}; 

  \begin{scope}[scale=1/3] % Scale the entire figure by 1/3
    % Define triangle size and style
    \tikzset{triangle/.style={draw=blue, thick, fill=none}}
    \tikzset{vertex/.style={circle, draw=blue, fill=blue, inner sep=0pt, minimum size=4pt}}
    % Define the points of the triangle
    \coordinate (A) at (0, 0);
    \coordinate (B) at (2, 0);
    \coordinate (C) at (1, 1.732);  % Height is sqrt(3) for equilateral triangle of side length 2
    % Draw the triangle
    \draw[triangle] (A) -- (B) -- (C) -- cycle;
    % Draw the vertices
    \node[vertex] at (A) {};
    \node[vertex] at (B) {};
    \node[vertex] at (C) {};
    \draw[blue, thick] (A) -- (-0.4,-0.1); \draw[blue, thick] (A) -- (-0.1,-0.4);
    \draw[blue, thick] (B) -- (2.4,-0.1); \draw[blue, thick] (B) -- (1.9,-0.4);
    \draw[blue, thick] (C) -- (1.2,2.1); \draw[blue, thick] (C) -- (0.8,2.1);
    \node at (1,-0.5) {\tiny{$3$}};
     \node at (0.2,1.1) {\tiny{$1$}};
      \node at (1.8,1.1) {\tiny{$2$}};
  \end{scope}
  \node at (2.35,0.3) { $ , \,\, e_1 \wedge e_2 \wedge e_3 \wedge \omega \Big] $}; 
\end{tikzpicture}
\end{figure}

\begin{figure}[h]
\begin{tikzpicture}
 \begin{scope}[shift={(4.7,0.3)}]
    \draw[blue,thick] (0,0) .. controls (0.333,0.4) and (0.666,0.4) .. (1,0);
    \draw[blue,thick] (0,0) .. controls (0.333,-0.4) and (0.666,-0.4) .. (1,0);
    \filldraw[blue] (0,0) circle (1.5pt);
    \filldraw[blue] (1,0) circle (1.5pt);
    \draw[blue,thick] (0,0) -- (-0.1,-0.1);
    \draw[blue,thick] (0,0) -- (-0.1,0);
    \draw[blue,thick] (0,0) -- (-0.1,0.1);
    \draw[blue,thick] (1,0) -- (1.1,-0.1);
    \draw[blue,thick] (1,0) -- (1.1,0);
    \draw[blue,thick] (1,0) -- (1.1,0.1);
    \node at (0.5,0.4) {{\tiny $2$}};
    \node at (0.5,-0.4) {{\tiny $3$}};
    \node at (2.3,0) {$  \, \, ,  e_2 \wedge e_3 \wedge \omega\Big]$};
    \node at (-0.8,0) {$ \,\,= -  \Big[$};
  \end{scope}
   \begin{scope}[shift={(9.1,0.3)}]
    \draw[blue,thick] (0,0) .. controls (0.333,0.4) and (0.666,0.4) .. (1,0);
    \draw[blue,thick] (0,0) .. controls (0.333,-0.4) and (0.666,-0.4) .. (1,0);
    \filldraw[blue] (0,0) circle (1.5pt);
    \filldraw[blue] (1,0) circle (1.5pt);
    \draw[blue,thick] (0,0) -- (-0.1,-0.1);
    \draw[blue,thick] (0,0) -- (-0.1,0);
    \draw[blue,thick] (0,0) -- (-0.1,0.1);
    \draw[blue,thick] (1,0) -- (1.1,-0.1);
    \draw[blue,thick] (1,0) -- (1.1,0);
    \draw[blue,thick] (1,0) -- (1.1,0.1);
    \node at (0.5,0.4) {{\tiny $1$}};
    \node at (0.5,-0.4) {{\tiny $3$}};
    \node at (2.3,0) {$  \, \, ,  e_1 \wedge e_3 \wedge \omega\Big]$};
    \node at (-0.6,0) {$ \,\, +  \Big[$};
  \end{scope}
  \begin{scope}[shift={(13.5,0.3)}]
    \draw[blue,thick] (0,0) .. controls (0.333,0.4) and (0.666,0.4) .. (1,0);
    \draw[blue,thick] (0,0) .. controls (0.333,-0.4) and (0.666,-0.4) .. (1,0);
    \filldraw[blue] (0,0) circle (1.5pt);
    \filldraw[blue] (1,0) circle (1.5pt);
    \draw[blue,thick] (0,0) -- (-0.1,-0.1);
    \draw[blue,thick] (0,0) -- (-0.1,0);
    \draw[blue,thick] (0,0) -- (-0.1,0.1);
    \draw[blue,thick] (1,0) -- (1.1,-0.1);
    \draw[blue,thick] (1,0) -- (1.1,0);
    \draw[blue,thick] (1,0) -- (1.1,0.1);
    \node at (0.5,0.4) {{\tiny $1$}};
    \node at (0.5,-0.4) {{\tiny $2$}};
    \node at (2.9,0) {$  \, \, ,  e_1 \wedge e_2 \wedge \omega\Big] + \ldots $};
    \node at (-0.6,0) {$ \,\, -  \Big[$};
  \end{scope}
\end{tikzpicture}
\[ = \qquad  0\qquad  +\qquad  0\qquad  +\qquad  0 \qquad + \qquad \cdots \qquad\qquad  \]
 \end{figure}
\noindent 

The terms corresponding to the contraction of edges $1,2,3$ give rise to graphs with double edges,   which vanish by the computation above. 

\end{ex} 

\begin{defn} Graph homology is defined to be the (infinite dimensional) vector space
\[ H(\GC_2) = \frac{\ker d}{\mathrm{Im} \, d}  \ .\]
The vector space $\GC_2$ is bigraded by the number of loops, and by the \emph{degree}\footnote{$\GC_N$ has degree $ |E_G|- N h_G$ and the complexes $\GC_N$ are essentially equivalent for all even $N$.    }
\[ \deg(G) =   |E_G| -  2 h_G \ .\]
The differential respects the number of loops, but decreases the degree by $1$. Thus $H(\GC_2)$ may be decomposed into homology groups of fixed degree
\[ H(\GC_2) = \bigoplus_{n}  H_n(\GC_2) \]
and each $H_n(\GC_2)$ in turn decomposes according to loop number  $ H_n(\GC_2) = \bigoplus_{g\geq 0}  H_n(\GC^{(g)}_2)$. 
\end{defn}
Each component $ H_n(\GC^{(g)}_2)$ is a finite-dimensional $\QQ$-vector space.

\begin{ex} Consider the  graph with 11 edges and 7 vertices depicted below on the left. The following calculation of the differential is valid for some appropriate choice of orientations on the corresponding graphs, which are omitted.

 \begin{figure}[h]
\begin{tikzpicture}

\node at (-1.3,0) {$d$}; 
  % Define the positions of the vertices
  \coordinate (A) at (-1, 0);
  \coordinate (B) at (0, 0);
  \coordinate (C) at (0, 1);
  \coordinate (D) at (0, -1);
  \coordinate (E) at (0.7, 0);
  \coordinate (F) at (1, 1);
  \coordinate (G) at (1, -1);

  % Draw the curved edges
  \draw[blue, thick, bend left=0] (A) to (B);
  \draw[blue, thick, bend left=40] (A) to (C);
  \draw[blue, thick,bend right=40] (A) to (D);
\draw[blue, thick,bend right=5] (B) to (C);
  \draw[blue, thick,bend left=5] (B) to (D);
  \draw[red, thick,bend left=5] (C) to (F);
    \draw[red, thick,bend right=5] (D) to (G);
      \draw[blue, thick,bend right=90] (G) to (F);
    \draw[blue, thick,bend left=5] (G) to (E);   
      \draw[blue, thick,bend right=5] (F) to (E);   
      \draw[red, thick,bend left=0] (E) to (B);   
       % Draw the vertices
  \foreach \coor in {A, B, C, D, E, F, G} {
    \filldraw[blue] (\coor) circle (2pt);
  }

\node at (5.9,0) {$-2$}; 

\begin{scope} [shift={(6.4,-0.3)}]
   % Define the positions of the vertices
  \coordinate (A) at (0, 0);
  \coordinate (B) at (0.5, 0.7);
  \coordinate (C) at (1, 0);
  \coordinate (D) at (1.5, 0.7);
  \coordinate (E) at (2, 0);
  \coordinate (F) at (2.5, 0.7);

  % Draw the curved edges
  \draw[blue, thick, bend left=0] (A) to (B);
  \draw[blue, thick, bend left=0] (B) to (C);
  \draw[blue, thick, bend left=0] (A) to (C);
  \draw[blue, thick, bend left=0] (C) to (D);
  \draw[blue, thick, bend left=0] (B) to (D);
  \draw[blue, thick, bend left=0] (D) to (E);
  \draw[blue, thick, bend left=0] (C) to (E);
   \draw[blue, thick, bend left=0] (E) to (F);
    \draw[blue, thick, bend left=0] (D) to (F);
    \draw[blue, thick, bend left=80] (F) to (A);
       % Draw the vertices
  \foreach \coor in {A, B, C, D, E, F} {
    \filldraw[blue] (\coor) circle (2pt);
  }\end{scope}

\node at (2.3,0) {$=$}; 
  \begin{scope}[shift={(4,0)}];

  % Draw the outer circle
\draw[blue, thick] (0,0) circle(1cm);
\fill[blue,thick] (0,0) circle(1pt);  
% Draw the spokes
\foreach \angle in {0, 72, 144, 216, 288} {
  \draw[blue, thick] (0,0) -- (\angle:1cm);
  \fill[blue,thick] (\angle:1cm) circle(2pt); 
}

  \end{scope}
\end{tikzpicture}
\end{figure}

\noindent 
Since every edge in the left-hand graph lies in a triangle except for the red edges, only contraction of these  three edges gives rise to a non-zero class in the graph complex. Indeed, contracting the red edge in the middle produces a wheel $W_5$; contracting the upper and lower red edges each  gives rise to  a zig-zag $Z_5$. 

\end{ex}
\begin{ex}Consider the  wheel graphs $W_n$ with $n\geq 3$ spokes introduced in the first lecture, \S\ref{sect: ExFeynmanSummary}. Since every edge in a wheel lies in a triangle, we have $d[W_n]=0$. 
\end{ex}

It is an exercise to check that every  wheel graph  $W_{2n}$ with an even number of spokes has an automorphism which induces an odd permutation on its set of edges. 
Consequently, it is  zero in $\GC_2$. 

\begin{thm} \cite{RossiWillwacher} The odd wheel classes are non-zero in graph homology:
\[   0\neq {[} W_{2n+1} ]   \in    H_{0}(\GC^{2n+1}_2) \]
for all $n\geq 1$.
\end{thm}
The proof of this theorem is transcendental, and involves computing an integral. We shall explain a second proof of this kind in the next lecture. 
Recently, Ben Ward \cite{Ward} has given  a purely algebraic proof of this theorem. Below is a picture of the homology of $\GC_2$\footnote{one of the primary reasons for considering $\GC_2$ rather than the possibly more natural choice $\GC_0$ is that the diagram fits more easily onto  the page}.

\begin{table}[h]
\begin{center}
\vspace{0.1in} 
\begin{tabular}{c|ccccccccccc}
$H_8$ &&&&&&&&&& $\0$   \\
$H_7$ &&&&&&&&& $\0$& 1   \\
$H_6$ &&&&&&&& $\0$ & 0& 0 \\
$H_5$ &&&&&&& $\0$ &  0 & 0& 0 \\
$H_4$ &&&&&&  $\0$ &  0 &  0& 0 &0\\
$H_3$ &&&&& $\0$ & $1$   & 0&  $1$ &1 &  2  \\
$H_2$ &&&& $\0$ & 0&0 & 0 & 0  & 0 & 0 \\
$H_1$ &&& $\0$ & 0 & 0 & 0 & 0 & 0 & 0 &  0 \\
$H_0$ && $\0$ & 1  & 0& 1 &0  & 1  &   $1 $ & 1 & 1 &  2\\ 
\hline  
$h_G$ & $1$  & $2$ & $3$ & $4$ & $5$ & $6$ & $7$ & $8$& $9$& $10$ & $11$
\end{tabular}
\vspace{0.2in} 
\end{center}
\caption{A number indicates the dimension of $H_k(\GC_2)$ at low loop order obtained by computer \cite{GraphComplexComputations, GraphComplexComputations24}.  All entries above the blue line of zeros necessarily vanish.  
 } 
\label{default}\label{table:knownresults} 
\end{table}

Graph homology has additional structure: it is a bigraded Lie coalgebra, i.e., admits a map $H(\GC_2) \rightarrow \bigwedge^2 H(\GC_2)$ satisfying certain properties. Its graded dual (known as graph cohomology) is a bigraded Lie algebra.

\begin{remark}
 The reader may recognise the numbers in the bottom row corresponding to $H_0$. They coincide with the conjectural dimensions of the  $\QQ$-vector space of multiple zeta values $\mathcal{Z}_g$ of weight $g$ modulo the ideal $I_g$ generated by   $\zeta(2)$ and non-trivial products.  
 
 A   choice of possible generators for these vector spaces are as follows:
\[ \begin{tabular}{c|ccccccccccc}
$g$ &  $2$ & $3$ & $4$ & $5$ & $6$ & $7$ & $8$& $9$& $10$ & $11$  \\
\hline
$\mathcal{Z}_g/I_g$ & $0$ & $\zeta(3)$  & $0$& $\zeta(5)$ & $0$  & $\zeta(7)$  & $\zeta(3,5)$&    $ \zeta(9) $ & $\zeta(3,7)$ & $\substack{ \zeta(11) \\ \zeta(3,3,5)}$ \\ 
\end{tabular} \ . 
\]

  \end{remark} 

\subsection{Some known results on the homology of the graph complex}
We briefly review some key theorems on  the homology of the graph complex. 
\begin{thm} (Willwacher \cite{WillwacherGRT}). The zeroth homology group $H_0(\GC_2)$ is isomorphic to the graded dual of the Grothendieck-Teichm\"uller Lie algebra $\mathfrak{grt}$ (equivalently, the graph cohomology group $H^0(\GC_2)$ is isomorphic, as a graded Lie algebra, to $\mathfrak{grt}$). 
\end{thm}
Unfortunately not much is known about the structure of $\mathfrak{grt}$ except that it contains the so-called `motivic' Lie algebra, which in turn governs the structure of motivic multiple zeta values. 

\begin{thm}  \label{thm: BrownFreeness} The Lie algebra $\mathfrak{grt}$  contains a copy of a free Lie algebra 
\[ \mathbb{L}(\sigma_3,\sigma_5,\ldots ) \hookrightarrow  \mathfrak{grt}\]
with one generator $\sigma_{2k+1}$ in every odd degree $2k+1$.  
\end{thm}
One can show that the elements $\sigma_{2k+1}$ are dual to the wheel classes $W_{2k+1}$. 
An open conjecture of Drinfeld's is that $\mathfrak{grt}$ should be isomorphic to the free graded Lie algebra with one element in every odd degree $\geq 3$, i.e., the previous map is surjective.

Let $\mathcal{M}_g$ denote the moduli space of curves of genus $g>1$. Its cohomology has a mixed Hodge structure, and in particular a weight filtration. 

\begin{thm} \cite{CGP} \[  H_n(\GC_2^{(g)}) \cong  H_{2g+n-1}(\LM_g^{\trop}) \cong \mathrm{gr}^W_{6g-6} H^{4g-6-n}( \mathcal{M}_g)\] 

\end{thm}

Putting the above results together implies that  the    top weight-graded pieces of the cohomology of the spaces  $\mathcal{M}_g$, for $g>1$,  contain the free Lie algebra $\mathbb{L}(\sigma_3,\sigma_5,\ldots )$.
This circle of ideas gives an application of the theory of multiple zeta values to topology.

It follows in particular that the dimension of  $\mathrm{gr}^W_{6g-6} H^{6g-6-n}( \mathcal{M}_g)$ grows exponentially fast in $g$. Recently, Borinsky has shown \cite{BorinskyEulerChar}, using methods of zero-dimensional quantum field theory, how to compute the Euler characteristic of $\GC_2$. He deduces that the homology of $\GC_2$ in fact grows super-exponentially. 

\subsection{Questions arising} 

\begin{quest} \label{quest1} What is the precise relationship between $H_0(\GC_2)$ and multiple zeta values? 
\end{quest} 

The  results  mentioned in  the previous paragraph imply  that $H_0(\GC_2)$ admits a map to the ring of  multiple zeta values modulo products and modulo $\zeta(2)$.
One expects that such a map should  arise by associating an integral to a graph, along the lines discussed in \S \ref{sect:FeynmanMZVs}.
We shall explain how one can achieve this in more detail  in the following lecture.
A related question is whether one can  associate motives to graph homology classes which would explain the appearance of the motivic Lie algebra  in theorem \ref{thm: BrownFreeness}. In degree zero, one would expect these to be mixed Tate motives over the integers. 

\begin{quest} \label{quest2} How can  we interpret the higher-degree homology $H_n(\GC_2)$ for $n>0$?
\end{quest}

If indeed one can assign motives and period integrals to graph homology classes, what kinds of numbers should we expect from the higher-degree homology classes?

\subsection{Differential forms on the moduli space of tropical curves}
A first answer to question \ref{quest1}  was given in the work of Rossi-Willwacher, who defined integrals associated to graphs of the type arising in deformation quantisation. They used this to show that the wheel classes are non-zero in graph homology.

 Since the graph complex computes the homology of the moduli space of tropical curves,  a more geometric approach would be to define integrals directly on the space $\LM_g^{\trop}$ of tropical curves, whose domains of integration  are the simplices $\sigma_G$. 
However, in order to define the integrands, one needs a notion of differential form on these spaces, which we recall are not differentiable  manifolds and, in addition, can have a local orbifold structure and therefore  the usual de Rham theory does not apply.

The following definition, however,  is adapted to these features.

\begin{defn} A differential $n$ form $\omega$ on $\LM_g^{\trop}$ is a family of differential forms 
\[ \{\omega_G\}_G  \ \in \  C^{n} ( \overline{\sigma}_G)  \]
which are smooth on an open neighbourhood of  $\overline{\sigma}_G$, one for each stable weighted  graph $G$ of genus $g$, which satisfy the compatibility conditions
\begin{eqnarray} 
\tau^* \omega_G & =&  \omega_G \qquad \hbox{ for all } \tau \in \mathrm{Aut}(G) \nonumber \\ 
\omega_G\Big|_{\overline{\sigma}_{G/e}} & =&  \omega_{G/e}    \nonumber 
\end{eqnarray}
where the second condition means that the restriction of $\omega_G$ to the locus $\ell_e=0$ (which, we recall,  is canonically identified with $\overline{\sigma}_{G/e}$)  coincides with $\omega_{G/e}$.
\end{defn} 

Thus a differential form on $\LM_g^{\trop}$ is a collection of forms $\omega_G$, one for each cell $\sigma_G$, which glue together in a consistent manner.

One can try to construct examples of such differential forms  recursively. Suppose that we have a differential form defined on all simplices of dimension $<k$. Then one can try to extend these forms to all simplices of dimension $k$, and proceed by induction. The existence  of such differential forms is  guaranteed by a kind of de Rham theorem for spaces of this kind \cite{brown2023bordifications}, but this procedure is not easy to carry out in practice.

Unfortunately, I do not presently know of an explicit general construction of differential forms  on $\LM_g^{\trop}$. There is, however, a very natural construction of differential forms which are allowed poles along the boundary locus $\partial \LM^{\trop}_g$ which we shall describe in the next lecture. 
Such a form is to be interpreted, roughly speaking, as a  differential form on  the open locus $\LM_g^{\circ, \trop}$.

\subsection{Example} \label{sect:exampleCanintW3}
The idea  is to use a  graph Laplacian matrix $\Lambda_G$ to define a differential form universally in $G$, in such a way that it does not depend on the choice of cycle basis for $H_1(G;\ZZ)$ which is implicit in the definition of  $\Lambda_G$. 

In the case of the wheel with 3 spokes, we shall consider in the next lecture  the \emph{canonical integral}
\[   \int_{\sigma_{W_3}}  \tr \left((\Lambda_{W_3}^{-1}  d\Lambda_{W_3})^3\right) = \int_{\sigma_{W_3}} 10 \frac{\Omega_{W_3}}{\Psi_{W_3}^2} = 60 \, \zeta(3)\ . 
\]
which happens in this case to reduce to a Feynman residue.   Feynman residues themselves will not do: the integrands cannot be assembled to form a compatible system as they stand  because the degree of the integrand of the  Feynman residue  of a graph  $G$ depends on the number of edges in $G$.

\section{Canonical forms and general linear group}

We now introduce a universal construction of differential forms associated to families of matrices, which will provide a canonical source of integrands compatible with all the structures considered so far. Applying this construction to the graph Laplacian matrices $\Lambda_G$ will give well-defined differential forms which can be assembled together in the manner of the previous lecture.   From this, we may  define `canonical' integrals on tropical moduli spaces of curves, which will enable us  to make a connection between graph homology classes, Feynman integrals and multiple zeta values.

\subsection{Bi-invariant forms}
Let $X$ be an $m \times m$ matrix whose entries $X_{ij}$ we think of as generic functions. The differential $dX$ is the matrix  whose entries are differential $1$-forms
\[ (dX)_{ij} = d X_{ij}\ . \]
Consider the differential $n$-form:
\[ \omega_X^n = \tr  \left(\left( X^{-1} dX \right)^n\right)   \quad \in     \QQ\!\left[X_{ij}, dX_{ij},  \frac{1}{\det(X)}\right]\ . \]
Since it involves inverting the matrix $X$ and has $n$ factors it \emph{a priori} has denominator $\det(X)^n$. 
One shows that for $n=1$ one has:
\[  \omega_X^1 =  d \log \det(X) =   \frac{ d \det(X)}{ \det(X)} \ .\]

\begin{ex}
Consider a generic family of $2\times 2$ matrices: 
\[ X= \begin{pmatrix} x_1 & x_3 \\ x_4 & x_2 \end{pmatrix} 
\]
The reason for the strange numbering is simply to make the signs easier to write down. One may check that 
\[ \omega^3_X = \tr (X^{-1} dX. X^{-1} dX . X^{-1} dX)  = 3  \sum_{i=1}^4 (-1)^i   \frac{  x_i 
\, dx_1 \wedge \ldots \widehat{dx_i} \wedge \ldots dx_4  }{ \det(X)^2}\]
\emph{A priori}, we expect  the power of $\det(X)$ in the denominator to be $3$, but there is a cancellation with a $\det(X)$ which factors into  the numerator. In the end, the power of $\det(X)$ in the denominator is only two.   
\end{ex}
A remarkable feature of these forms is that, despite their definition, significant cancellations occur, resulting in much lower powers of $\det(X)$  in the denominator than one might expect \emph{a priori}.

\begin{ex} \label{ex: omega5rank3} 
Consider the symmetric matrix
\[ X = \begin{pmatrix}   
x_1 & x_4 & x_5 \\
 x_4 & x_2 & x_6 \\
 x_5 & x_6 & x_3 
 \end{pmatrix}\ .
 \] 
 Then one finds that $\omega^3_X=0$, and with the help of a computer, one may check that 
 \begin{equation}  \omega^5_X=  10 \sum_{i=1}^5  (-1)^i  \frac{ x_i\, dx_1 \wedge \ldots \wedge \widehat{ dx_i} \wedge \ldots \wedge dx_{6} }{ \det(X)^2} 
 \end{equation}
 Again, one sees that there are many    cancellations: the power of $\det(X)$ in the denominator is only two, whereas one would have expected it to be five.
\end{ex}

\subsection{Properties}
The forms $\omega_X^n$ have a number of remarkable properties.  One shows as a simple  consequence  of the definition that they vanish if $n$ is even, and define closed forms otherwise: $d \omega_X^n=0$.  Furthermore, one has
\[ \omega^n_{X^T} = (-1)^{\frac{n(n-1)}{2}}\omega^n_X\]
which implies that if $X$ is symmetric, then $\omega^n_X$ vanishes for $n\equiv 3 \mod 4$. 
Finally, if $n>1$ one has the projective invariance property
\[ \omega^n_{\lambda X} = \omega^n_X\]
for any $\lambda$ (not necessarily constant). Note that this property fails for $\omega^1_X$, which satisfies $\omega^1_{\lambda X } = \omega^1_X + \frac{d \lambda}{\lambda}$ and for this reason will be discarded.

We will typically consider symmetric matrices $X$. Thus, after setting aside $\omega^1_X$,  we have at our disposal an infinite sequence of closed differential forms:
\[ \omega^5_X , \omega^9_X, \omega^{13}_X, \ldots \]
The most important property of these forms are their  \emph{bi-invariance}.

\begin{lem} Let $X$ be an invertible  $m\times m$ matrix as above.  For any invertible matrices $A, B \in \GL_m(\RR)$ with constant entries (i.e. $dA= dB=0$), one has
\[ \omega^n_{X}= \omega^n_{AX} = \omega^n_{XB}\ .\] 
\end{lem} 
\begin{proof} The form $X^{-1} dX$ is already invariant by left-multiplication by $A$. Indeed, 
\[  (AX)^{-1} d(AX) =  X A^{-1} A \, dX = X^{-1} dX\ .\]
It follows that $\omega^n_X$ is also left-invariant. Now observe that  $dX X^{-1}$ is invariant by right-multiplication by $B$ and notice that
\[  \omega^n_X = \tr (X^{-1} dX. \ldots X^{-1} dX) = \tr (  dX X^{-1}. \ldots dX X^{-1})\]
by cyclicity of the trace.  
\end{proof} 
The property of bi-invariance implies that 
\[ \omega^n_G := \omega^n_{\Lambda_G}\]
is well-defined for any choice of graph Laplacian matrix $\Lambda_G$ associated to $G$. Indeed, changing cycle basis in the definition of $\Lambda_G$ modifies $\Lambda_G$ via the transformation 
\[ \Lambda_G \mapsto P^T \Lambda_G P \qquad \hbox{ for some } P \in \GL_{h_G}(\ZZ) \ . \] 
This implies that the forms $\omega_G$
 associated to graph Laplacians are well-defined independently of the choice of cycle basis, and consequently glue consistently across the cells of the tropical moduli space.
 
\subsection{Canonical integrals}
Define the \emph{canonical algebra} to be the  graded exterior algebra 
\[ \Omega^{\bullet}_{\mathrm{can}} = \bigwedge \left( \bigoplus_{k>0} \QQ \omega^{4k+1}\right) \]
to be viewed as a formal graded algebra where $\omega^{4k+1}$ is in degree $4k+1$. 
 It is generated as a $\QQ$-vector space by 
\[  \omega^{4i_1+1} \wedge \omega^{4i_2+1} \wedge \ldots \wedge \omega^{4i_r+1}\]
for all increasing sequences $1\leq i_1< i_2< \ldots< i_r$.  Each element $\omega \in \Omega_{\mathrm{can}}$ may  be viewed as a map which assigns to any graph  $G$ a differential form $\omega_G$.

 \begin{defn}
 Let $G$ be any connected graph with $n+1$ edges and let $\omega \in \Omega^{n}_{\mathrm{can}}$ 
 be of degree $n$. Define the \emph{canonical integral} to be 
 \begin{equation} \label{defn: Icanonical} I_G( \omega)  = \int_{\sigma_G}  \omega_G \end{equation}
\end{defn}
 Unlike Feynman residues, which may diverge due to boundary singularities,  a striking feature of  canonical forms is that they automatically compensate these divergences.
 \begin{thm} \cite{BrSigma} \label{thm: CanIntisfinite} The canonical integrals $I_G(\omega)$ are always finite.
 \end{thm}

 Thus every canonical form $\omega\in \Omega^n_{\mathrm{can}}$ defines a map 
 \[ I_{\bullet}(\omega):  \{\hbox{Connected graphs with $n+1$ edges}\} \To \RR\]
 For any such $\omega$, it follows from the definition $\omega_G = \omega_{\Lambda_G}$ that 
 \[  \omega  \  \in \  \QQ\left[ x_e, dx_e, e\in E_G \right] \left[ \frac{1}{\det(\Lambda_G)}\right] \ . \]
 Recall that $\det(\Lambda_G) = \Psi_G$ is the graph polynomial. Thus we may write
 \begin{equation}\label{IcanisgeneralisedFeynman} I_G(\omega)= \int_{\sigma_G}  \frac{N_G}{\Psi_G^k} \Omega_G
 \end{equation}
 for some integer $k>0$ and $N_G\in \QQ[x_e, e\in E_G]$ is  a certain numerator polynomial in the edge variables $x_e$. It has  the magical quality that it compensates all the  singularities arising from the denominator $\Psi_G$, which vanishes along certain boundary components of the domain of integration, as discussed earlier.  
Thus canonical integrals may be viewed as natural “regularisations” of Feynman integrals, in which the numerator is determined so as to cancel all boundary divergences.
 
 In short, \emph{canonical integrals are generalised Feynman integrals with the extraordinary property that they are always finite, irrespective of the graph $G$}.
 
 \subsection{Examples} In the following examples, we shall compare and contrast  canonical integrals with Feynman residues, which we recall are given by the formula 
 \[ I^{\res}_G = \int_{\sigma_G} \frac{\Omega_G}{\Psi_G^2}\]
 and which are often infinite.
 These examples illustrate that canonical integrals are closely related to Feynman residues, but may differ by subtle correction terms which alter both convergence properties and the weights of the resulting periods. 
A key point is   that canonical integrals associated to one graph may coincide with Feynman residues of a different graph, suggesting the existence of hidden equivalences between these integrals.

 \begin{enumerate}
 \item Let $G= W_3$ be the  wheel with 3 spokes. By the computation in example \ref{ex: omega5rank3}, the canonical form is proportional to the Feynman residue integrand in this case. Thus the canonical integral is a multiple of the Feynman residue,  see \S \ref{sect:exampleCanintW3}. \\
 
 \item  Let $G= W_5$ be the wheel with 5 spokes. Then with the help of a  computer, one checks that the canonical form simplifies to two terms:
 \[  \omega^9_{W_5} =  18 \left( \frac{\Omega_{W_5}}{\Psi_G^2} +   12 \, \frac{x_1x_2\ldots x_5}{\Psi_{W_3}^3} \Omega_{W_5} \right)  \,  \]
 where the edge variables $x_1,\ldots, x_5$ correspond to the five spokes of the wheel $W_5$. The first term is the Feynman residue form and its integral is finite:
 \begin{equation} \label{residueW5} I^{\res}_{W_5} = 70 \zeta(7) \ . \end{equation} 
 The second term  in brackets also converges and yields
 \[  \int_{\sigma_{W_5}}  12 \,  \frac{x_1x_2\ldots x_5}{\Psi_{W_3}^3} \Omega_{W_5} = 70 (\zeta(5) - \zeta(7))\ . \]
 Thus the $\zeta(7)$ term cancels out of the canonical integral and one obtains
 \begin{equation}  \label{canonicalW5} I_{W_5} (\omega^9) =  1260 \, \zeta(5)\ .\end{equation}
 Notice that the canonical integral \eqref{canonicalW5} and Feynman residue \eqref{residueW5} for $W_5$ have different weights ($5$ and $7$ respectively).
 
 \begin{remark} It would be interesting to show  directly that the canonical integral is proportional to the Feynman residue for the wheel with \emph{four} spokes: 
 \[ I_{W_5} (\omega^9) =  18\,  I^{\res}_{W_4} \] 
 Thus the canonical integral is in fact a Feynman residue, but for a different graph! A more striking example of this phenomenon is discussed below.
 \end{remark}

\item  For general canonical integrals of the canonical forms $\omega_G^{4k+1}$ one can show that $I_G(\omega^{4k+1})$ vanishes unless 
$G$ has exactly $h_G= 2k+1$ loops.   We extend the definition of $I_{\bullet}(\omega)$ linearly to define a map on linear combinations of graphs.

\begin{thm} Let $\Xi \in \GC_2$ be a linear combination of oriented graphs with $g$ loops and  $4k+2$ edges whose differential vanishes: $d \,\Xi=0.$ Then if the canonical integral $I_{\Xi}(\omega^{4k+1})$ is non-zero,  $[\Xi] \in H_0(\GC_2)$ is non-zero in homology.
\end{thm} 
It follows that the canonical integrals can detect graph homology cycles. \\

 \item (General canonical wheel integrals).
 
 \begin{thm} \label{thm43}  \cite{BrownSchnetzoneformmatrices}  The canonical wheel integrals for $g\geq 3$ odd are
 \[ I_{W_g}(\omega^{2g-1}) =  g \, \binom{2g}{g} \zeta(g)\ .\]
   \end{thm} 

   The differential form $\omega_{W_g}^{2g-1}$ has the form 
   \begin{equation} \label{shapecanonicalwheels} \sum_{k\geq 0}   c_k \left(\frac{x_1\ldots x_g}{\Psi_G}\right)^k  \frac{\Omega_G}{\Psi_G^2} \end{equation}
   for certain coefficients $c_k \in \QQ$, where $x_1,\ldots, x_g$ are the edge variables corresponding to the internal spokes of $G$. It is a generalised Feynman integral, but has lowest possible weight amongst integrals of the form \eqref{shapecanonicalwheels} compared to the Feynman residues of the wheel graphs, which have highest possible weight amongst integrals of the form \eqref{shapecanonicalwheels}. 
The previous theorem  has a number of geometric consequences, including the fact that the wheel classes $[W_g]$ are non-zero in $H_0(\GC_2)$.  
\\

More generally, Jean-Luc Portner  has proven the following result in his thesis.

\begin{thm} (Portner) Let $G$ be a connected graph with  $4k+2$ edges. Then the  canonical integral 
\[  I_G(\omega^{4k+1}) \]
is a (single-valued) multiple zeta value of weight $2k+1$.
\end{thm} 
In fact he proves that the canonical integral is proportional to the Rossi-Willwacher integrals, which gives another proof of  theorem \ref{thm43}. 

 \end{enumerate}
 
 Next we turn to non-trivial exterior products of two or more canonical forms.
 
 \begin{ex} Let $G= K_6$ denote the complete graph on six vertices. It has  15 edges. 
 One can check with a computer that 
 \[ \omega_{K_6}^5 \wedge \omega_{K_6}^9 =  \frac{9!}{8} \, \frac{\prod_{e\in E_{K_6}} x_e}{\Psi_{K_6}^3}  \, \Omega_{K_6}\ .\]
 This graph has a huge number of divergent subgraphs and is therefore as bad as one can get from the point of view of the Feynman residue. Nevertheless, the integrand above is convergent and Borinsky and Schnetz were able to compute the canonical integral:
 \begin{equation} \label{IK6} I_{K_6}(\omega^5 \wedge \omega^9) = \frac{9!}{16} \left( 360\,  \zeta(3,5) + 690 \zeta(3) \zeta(5) - \frac{29}{315} \pi^8 \right)   \end{equation}
 \[ = \frac{9!}{16} \left( 15\, \zeta(3) \zeta(5)  - \frac{25}{96} I^{\res}_{K_{3,4}}\right) \ ,\]
 where $I^{\res}_{K_{3,4}}$ was the Feynman residue for the bipartite graph mentioned in Lecture 1. Thus once again we see that the canonical integral is closely related to a Feynman residue of a completely different graph. It is a challenge to prove this directly (for example, by integration by parts and application of Stokes' theorem \cite{BrSigma}).  
 \end{ex}

 Recall from lecture 1 that we \emph{do not} expect  all Feynman residues to be multiple zeta values. 
  An obvious  question  can be vaguely phrased as follows: 
 \begin{quest}
 Are all canonical integrals MZV's, or are they more like Feynman residues (in which case they are not MZV's)? 
 \end{quest} 
 Note that there is a difference between evaluating a canonical integral for \emph{individual graphs} as opposed to linear combinations of graphs which are closed in the graph complex $\GC_2$. The former class of numbers is larger than the latter so the corresponding answers to the previous question may be different in both cases.

 \begin{ex} We know of an infinite family of non-explicit canonical integrals of wedge products $\omega^{4i+1} \wedge \omega^{4j+1}$. The
  statement is as follows: 
 
 \begin{thm}  \cite{BrownSchnetzoneformmatrices} For all odd integers $3\leq m<n$ there exists a non-trivial  closed element  $\Xi_{m,n} \in \GC_2$ with edge degree $2m+2n-1$, such that
 \[  I_{\Xi_{m,n}} (\omega^{2m-1} \wedge \omega^{2n-1}) = \zeta(m) \zeta(n)\]
 and $[\Xi_{m,n}]  \in H_{k_{m,n}}(\GC_2)$ is non-zero in graph homology in some odd  degree $k_{m,n}$ where  $0<k_{m,n}<  2 \min\{m,n\} -3.$
 \end{thm} 
 We do not know how to pin down the  precise degree $k_{m,n}$.
 \end{ex} 
 
 \subsection{Picture} We can now revisit the earlier diagram of graph homology (or rather its dual, graph cohomology) and match non-zero classes with a canonical form whose corresponding integral is non-trivial.  Since graph cohomology has a Lie algebra structure, we may take Lie brackets of such elements to obtain potentially more classes.  See table \ref{table:knownrevisited}.

\begin{table}[h]
\begin{center}
\vspace{0.1in} 
\begin{tabular}{c|ccccccccccc}
$H^8$ &&&&&&&&&& $\0$   \\
$H^7$ &&&&&&&&& $\0$& $\omega^9\wedge \omega^{17}$   \\
$H^6$ &&&&&&&& $\0$ & 0& 0 \\
$H^5$ &&&&&&& $\0$ &  0 & 0& 0 \\
$H^4$ &&&&&&  $\0$ &  0 &  0& 0 &0\\
$H^3$ &&&&& $\0$ & $\omega^5 \wedge \omega^9$   & 0&  $\omega^5 \wedge \omega^{13}$ & $[\omega^5, \omega^5\wedge \omega^9]$ &  $ \substack{ [\omega^5,\omega^{17}] \\  [\omega^9,\omega^{13}] } \}$    \\
$H^2$ &&&& $\0$ & 0&0 & 0 & 0  & 0 & 0 \\
$H^1$ &&& $\0$ & 0 & 0 & 0 & 0 & 0 & 0 &  0 \\
$H^0$ && $\0$ & $\omega^5$  & 0& $\omega^9$ &0  & $\omega^{13}$ &   $[\omega^5,\omega^9] $ & $\omega^{17}$ &  $[\omega^5,\omega^{13}] $ &   $\substack{ \omega^{21} \\  [\omega^5,[\omega^5,\omega^9]]} \}$\\ 
\hline  
$h_G$ & $1$  & $2$ & $3$ & $4$ & $5$ & $6$ & $7$ & $8$& $9$& $10$ & $11$
\end{tabular}
\vspace{0.2in} 
\end{center}
\caption{Elements of graph cohomology $H^k(\GC_2)$  revisited and labelled by a corresponding Lie bracket of canonical forms.
 } 
\label{default}\label{table:knownrevisited} 
\end{table}
 \begin{quest} Write down an explicit differential form on $\LM_8^{\trop}$ which represents the class $[\omega^5,\omega^9]$.  It should correspond to an automorphic differential form on $\GL_8(\ZZ)$ (see next chapter), and its associated canonical integral over the class in $H_0(\GC_2^{(8)})$ should produce an irreducible MZV of weight $8$ (such as  $\zeta(3,5)$). 
 \end{quest}

 \subsection{Comments}
 We have now come full circle and shown how   the seemingly disparate ideas of these four lectures are in fact  closely interconnected. 

We have so far encountered two \emph{a priori} distinct  non-linear representations  of MZVs: on the one hand, Feynman residues, and on the other,  canonical integrals. They appear  to be related in a subtle way.
The  former arise in high-energy physics, where their mathematical interpretation remains far from clear. The latter are related to the cohomology of the even graph complex, and may be viewed as kinds of `volumes' of cells in a moduli space of tropical curves. The Stokes relations on these spaces  seem to relate the two, and indeed  both are  special cases of `generalised Feynman periods'.   
This points towards a geometric interpretation of Feynman periods in quantum field theory that can be formulated entirely in mathematical terms, without reference to the underlying physical theory.

\section{Epilogue: Panorama and questions}

This final section is slightly more advanced, and goes beyond what was covered in my four lectures.
We explain how the constructions introduced above arise naturally in the context of the general linear group and its associated symmetric spaces, thereby linking the entire framework to classical results in arithmetic geometry. Along the way we discuss two more  `non-linear' interpretations for MZV's.

We first discuss the tropical Torelli map  which associates to a metric graph its graph Laplacian, thereby embedding a certain part of  the moduli space of tropical curves into the symmetric space of positive definite matrices modulo the action of $\mathrm{GL}_g(\ZZ)$. 
Under this map, the canonical forms on matrices pull back to the forms constructed earlier on tropical moduli spaces, showing that the latter are manifestations of invariant differential forms on locally symmetric spaces.
From this perspective,  multiple zeta values arise as periods associated to the cohomology of arithmetic groups.
\subsection{Tropical Torelli map, general linear group and algebraic K-theory}

The canonical integrals have a natural geometric interpretation. There is a map, called the tropical Torelli map:
\[ \lambda:  \mathcal{M}_g^{\circ, \trop}  \To P_g/\GL_g(\ZZ)\]
which sends a marked metric graph $G$ to the class ${[}\Lambda_G]$ of the graph Laplacian

The space  $P_g$ is the space of positive-definite symmetric  matrices $X$, where $P\in \GL_g(\ZZ)$ acts via
\[   X \mapsto P^T  X P \ .\]
Thus the equivalence class ${[}\Lambda_G]$  in the quotient $P_g/\GL_g(\ZZ)$ is well-defined\footnote{The tropical Torelli map may be interpreted as the tropicalisation of the map which to a compact Riemann surface  $C$  of genus $g$ associates its Riemann polarisation form. This is the $g \times g$ matrix whose $i,j$th entry is 
\[  \frac{1}{2\pi i}  \int_{C(\CC)} \omega_i \wedge \overline{\omega}_j\] 
where $\omega_1,\ldots, \omega_g$ are a basis for the holomorphic differentials  on $C$. }.

The canonical forms $X\mapsto \omega_X$, for any $\omega \in \Omega_{\mathrm{can}}$  define closed differential forms on $P_g$ which are invariant for the action of $\GL_g(\ZZ)$.  Borel famously computed the stable cohomology of the $\GL_g(\ZZ)$ in the 1970's. His result may be rephrased as saying that the natural map
\begin{equation}\label{stablecohom} \Omega^{\bullet}_{\mathrm{can}}\otimes_{\QQ} \RR \To  \varprojlim_g  H^{\bullet}( P_g/\GL_g(\ZZ); \RR)\end{equation}
is an isomorphism. In other words, the stable cohomology of the general linear group $\GL_g(\ZZ)$ is exactly given by the canonical bi-invariant forms. This builds on important prior work of Matsushima and Garland;  Borel's contribution was to prove the injectivity of \eqref{stablecohom} in very small degrees by studying the asymptotic growth of differential forms as they approach the boundary of $P_g$\footnote{A much stronger version of this result \cite{brown2023bordifications} can be proven using the ideas discussed in these notes}.  An important corollary of Borel's theorem is his computation of the rank of the algebraic $K$-theory of the integers:
\[  K_{n}(\ZZ)\otimes_{\ZZ} \QQ \cong  \begin{cases}   \QQ \quad  \hbox{  if } n\equiv 1 \pmod 4 \hbox{ and } n>1 \\ 
 0 \qquad  \hbox{ otherwise}  \end{cases}\]  
 This result is the lynchpin upon which the entire motivic theory of multiple zeta values depends: it explains the fact that the relations between multiple zeta values are governed by a free Lie algebra in odd degrees $\sigma_3,\sigma_5,\ldots$, which are in one-to-one correspondence  with the forms $\omega^5, \omega^9 ,\ldots $.

 The  canonical forms associated to graphs   are simply the pull-backs of the bi-invariant forms, which naturally live on $P_g/\GL_g(\ZZ)$, by the Torelli map:
 \[ \omega_G =  (\lambda^* \omega)\big|_{\sigma_G}\]
 for any $\omega \in \Omega^n_{\mathrm{can}}.$
 
 \subsection{* Volumes, canonical Voronoi integrals, regulators}
 We briefly discuss the non-linear representations of zeta values  as volumes of locally symmetric spaces and regulators, as mentioned in the  introduction.
 
\subsubsection{Volumes} First, we discuss a classical computation  due to Minkowski, who  calculated the covolume of $\mathrm{GL}_g(\ZZ)$ inside $\mathrm{GL}_g(\RR)$ as a product of consecutive zeta values.  \begin{thm} Let $g>1$ be odd. Then 
 \[ \mathrm{vo}l_g =  \omega^5 \wedge \omega^9 \wedge \ldots \wedge \omega^{2g-1}\]
 is a non-zero $\GL_g(\RR)$-invariant   volume form on $\RR^{\times}_{>0} \backslash P_g$. There exists a non-zero rational number $\alpha_g$ such that the volume 
 \[ \int_{\RR^{\times}_{>0} \backslash P_g/\mathrm{GL}_g(\ZZ)}  \mathrm{vol}_g =  \alpha_g \, \zeta(3) \zeta(5) \ldots \zeta(g)\]
 \end{thm} 
 
 The first part of this theorem uses results on invariant theory to show that the exterior product of the canonical forms $\omega^5 \wedge \ldots \wedge \omega^{2g-1}$ is indeed non-zero. The second part can be deduced  from Minkowski's theorem.   In particular, we deduce yet another non-linear integral representation for $\zeta(3)$:
 \[ \int_{\RR^{\times}_{>0} \backslash P_3/\mathrm{GL}_3(\ZZ)}  \omega^5 =  \alpha_3 \, \zeta(3) \ .\]
 It turns out to be completely equivalent to the canonical integral $I_{W_3}(\omega^5)$ and the Feynman residue $I^{\mathrm{res}}_{W_3}$. 
  The volume computation of Minkowski was generalised by Siegel to all number fields.

\subsubsection{Vorono\"i decomposition} Much of the story we have discussed for the moduli space of tropical curves can be adapted to the setting of the locally symmetric space  $\RR^{\times}_{>0} \backslash P_g/\GL_g(\ZZ)$ for the general linear group. It too admits a cellular decomposition, due to Vorono\"i, which is analogous to the decomposition of  $\LM_g^{\trop}$ into simplices $\sigma_G$. 
 
 We briefly discuss some key ideas  and refer to the excellent paper \cite{TopWeightAg} for further details.
 
 \begin{defn} Vorono\"i's construction of a cellular decomposition of $P_g$ is as follows. Each cell is indexed by a positive definite quadratic form 
 \[ Q(x,y)  =  \sum_{1\leq i,j\leq g} a_{ij} x_i y_j  \]
 where $x, y\in \RR^g$ are column vectors of length $g$. It may be written $Q(x,y) = x^T A y$ where $A=(a_{ij})_{ij}$ is a positive definite symmetric matrix (i.e., 
 one has $Q(x,x) \geq 0$ with equality if and only if $x=0$).  To any such $Q$ one writes down the set of minimal vectors
 \[ M_{Q} = \{ \xi \in \ZZ^g \setminus \{0\}:  Q(\xi, \xi) \leq  Q(z, z) \quad  \hbox{ for all } z \in \ZZ^g\setminus \{0\} \}\]
 which is a finite set of vectors in $\ZZ^g$. The cell associated to $Q$ is then the convex hull 
 \[ C_{Q} =  \RR_{>0}  \langle \xi \xi^T : \xi \in M_{Q} \rangle\]
 where $\xi \xi^T$ are the symmetric rank 1 matrices whose $i,j$th entries are $\xi_i\xi_j$, where $\xi_i$ are the coordinates of a given minimal vector $\xi\in M_Q$.
 We denote by 
 \[ \sigma_Q = \RR^{\times}_{>0} \backslash C_Q\]
 the corresponding quotient by the multiplicative scalar action of $\RR_{>0}$. 
  \end{defn} 
 \begin{ex} Consider the quadratic form for $g=2$ given by $Q(x,x) = x^2_1+2x_1x_2 + x_2^2$. 
 It has six    minimal vectors 
 \[  M_{Q} =  \{ (x_1,x_2) =   ( \pm 1, 0)  ,   \pm (1,-1)  ,  (0, \pm 1) \} \ .\]
 The corresponding symmetric matrices $\xi\xi^T$ for $\xi \in M_{Q}$ are of the form
 \[  \left\{  \begin{pmatrix}  1 & 0 \\ 0 & 0 \end{pmatrix} \quad , \quad  \begin{pmatrix} 1 & -1 \\ -1 & 1 \end{pmatrix}  \quad , \quad    \begin{pmatrix}  0 & 0 \\ 0 & 1 \end{pmatrix}
  \right\}\ . \]
 Their convex hull consists of linear combinations of these three matrices with positive real coefficients:
\[ C_{Q} = \left\{  \begin{pmatrix}  \lambda_1+\lambda_2 & -\lambda_2 \\  -\lambda_{2}  & \lambda_1+\lambda_3 \end{pmatrix} : \hbox{ for } \lambda_1,\lambda_2,\lambda_3>0    \right\} \] 
The above matrix is nothing other than the graph Laplacian for the sunrise graph 
\eqref{LambdaSunrise}.

 \end{ex}

The set of cells $\sigma_Q$ as $Q$ ranges over all positive definite quadratic forms $Q$ decomposes the space $\RR^{\times}_{>0} \setminus P_g$ into polyhedral cells with the property that any distinct two such cells are either disjoint, or meet in a common face. Furthermore, this decomposition is equivariant with respect to the action of $\GL_g(\ZZ)$ and thus, one may show, gives  a finite cell decomposition of  $\RR^{\times}_{>0} \setminus P_g/\GL_g(\ZZ)$. 

One can prove \cite{AlexeevBrunyate} that the tropical Torelli map sends cells $\sigma_G$ onto cells of the form   $\sigma_Q$.
%More precisely, to every cell $\sigma_G$ in $\LM_g^{\trop}$ there exists a quadratic form $Q_{G}$ such that $\lambda: \sigma_G \rightarrow \sigma_{Q_G}.$ 
The following theorem is a generalisation of theorem \ref{thm: CanIntisfinite}.

\begin{thm}  \cite{brown2023bordifications} \label{thm: IQomegafinite} Let $\omega\in \Omega^{d}_\mathrm{can}$ and let $\sigma_Q$ be a Vorono\"i  cell of dimension $d-1$. Then  the `canonical Vorono\"i' integral  
\[ I_Q(\omega) = \int_{\RR_{>0}^{\times} \setminus  \sigma_Q} \omega\]
is finite. 
\end{thm}

   Let $g>1$ be odd. Recall that $\mathrm{vol}_g = \omega^5 \wedge \ldots \wedge  \omega^{2g-1}$ is a  volume form. 
By decomposing the fundamental domain into top-dimensional Vorono\"i cells $\sigma_{Q_1},\ldots, \sigma_{Q_N}$ we deduce that 
\[   \alpha_g \zeta(3) \zeta(5) \ldots \zeta(g)  = \int_{\RR^{\times}_{>0} \backslash  P_g/\GL_g(\ZZ)} \mathrm{vol}_g = \sum_{i=1}^N  I_{Q_i} (\mathrm{vol}_g)\]
where the left-hand equality follows from Minkowski's theorem. Thus the product of zeta values may be broken up as a sum of volumes of Vorono\"i cells. 

\begin{quest}
    What kinds of numbers can appear as volumes  $I_{Q_i}(\mathrm{vol}_g)$?
\end{quest}

\begin{ex} For $g=3$, the space $\RR^{\times}_{>0} \backslash  P_3/\GL_3(\ZZ)$ has a single cone in its Vorono\"i decomposition, which is the image of the wheel $W_3$.
Thus we retrieve, once more, the fact that the wheel canonical integral is proportional to $\zeta(3)$.

The space $\RR^{\times}_{>0} \backslash  P_5/\GL_5(\ZZ)$ has exactly 3 cones in its Vorono\"i decomposition.  See \cite{TopWeightAg, ElbazVincentGanglSoule} for a precise description. One of them  is the image of the cell $\sigma_{K_6}$ corresponding to the complete graph with $6$ vertices.  The fact that its canonical integral involves a non-irreducible multiple zeta value of weight 8 (namely $\zeta(3,5)$, see \eqref{IK6})  gives an arithmetic way to see that it cannot be the only such cell. 

What are the volumes of the two other cones? 
\end{ex}

 \subsubsection{Regulators} For all $k\geq 1$,  Borel defined the  regulator
 \[ r_{4k+1} : K_{4k+1}(\ZZ) \otimes_{\ZZ} \QQ \To \CC\] 
and proved that its image is the line $\QQ \zeta(2k+1)$. It follows from our earlier interpretation of the stable cohomology of the general linear group that it is computed by the integral of the canonical form $\omega^{4k+1}$ over a certain locally finite homology class for $\GL_g(\ZZ)$ where $g$ is sufficiently large (in the stable range). In particular, it is a linear combination of integrals of the form described   in theorem \ref{thm: IQomegafinite}.  This computation is  not only the origin of the non-triviality of the `motivic' Soule elements $\sigma_{2k+1}$, but also the reason for the appearance of  odd zeta values in the theory of mixed Tate motives.

 \begin{remark}A subtle fact  is that regulators are periods of the \emph{stable} cohomology of $\GL_g(\ZZ)$, and not immediately related to the wheel classes, which are dual to \emph{compactly supported} cohomology classes, and which in any case lie outside the stable range. There is a subtle way to relate the  two which is explained in \cite{BrownSchnetzoneformmatrices} \S2.4. For example, the first regulator map on $K_5$ may be shown to be the quotient of a volume integral by a wheel integral
 \[  \mathrm{Im} (r_5) =  \frac{ \mathrm{vol}(  \RR^{\times}_{>0} \backslash  P_7/\GL_7(\ZZ))}{I_{W_5}} \QQ =\frac{\zeta(3) \zeta(5)}{\zeta(5)} \QQ = \zeta(3)  \QQ \]

 \end{remark}

 \subsection{Concluding remarks and  questions}

There are two main sources for the appearance of multiple zeta values and mixed Tate motives in mathematics and physics. The first arises from linear geometry (relating to the fact that  projective spaces have  cohomology of Tate type) and appears in the context of hyperplane configurations,  and in particular   moduli spaces of curves of genus~$0$ with $n$ marked points $\mathcal{M}_{0,n}$. At the core of this theory lies the unipotent  fundamental group of the projective line minus three points, $\mathbb{P}^1\setminus \{0,1,\infty\}$~\cite{deligneP1}.

The second arises from a `non-linear' geometry associated with  singular determinant loci. As we have explained, these arise in the context of Feynman integrals in quantum field theory, as canonical integrals on   moduli spaces of tropical curves, and  on the locally symmetric spaces of the general linear group $\GL_n(\ZZ)$.  The latter are closely related  to the theory of regulators in  algebraic $K$-theory. We expect that all multiple zeta values (or certainly a generating family modulo $\zeta(2)$) naturally arise as  integrals on these spaces, and that, in a suitable sense, the category of mixed Tate motives over $\ZZ$ can be generated geometrically   from a cell decomposition of the locally symmetric spaces  for   $\GL_n(\ZZ)$.
 
The examples discussed in these notes suggest that multiple zeta values are governed by a common geometric framework based on determinantal structures, which encompasses both the classical linear theory and the non-linear constructions arising in quantum field theory and arithmetic geometry. A central problem is to make this framework explicit, by identifying a direct geometric relationship between the linear and non-linear representations and thereby clarifying the unity underlying these apparently different constructions.

\renewcommand\refname{References}

\bibliographystyle{alpha}

\bibliography{biblio}

\end{document}